\documentclass[11pt,amssymb]{amsart}
\usepackage{amssymb}
\usepackage[mathscr]{eucal}
\usepackage[all,cmtip]{xy}
\newcommand{\uC}{\underline{C}}
\newcommand{\Z}{{\mathbb Z}}

\newcommand{\Q}{{\mathbb Q}}
\newcommand{\I}{\mathbb{I}}

\newcommand{\Fpm}{\mathrm{(F}_m'\mathrm{)}}

\newcommand{\R}{{\mathbb R}}

\newcommand{\Br}{\mathrm{Br}}

\newcommand{\fX}{\mathfrak{X}}

\newcommand{\Ga}{\mathrm{Gal}}
\newtheorem{thm}{Theorem}[section]
\newtheorem{lemma}[thm]{Lemma}
\newtheorem{prop}[thm]{Proposition}
\newtheorem{cor}[thm]{Corollary}

\newcommand{\gen}{\mathbf{gen}}

\newcommand{\Pic}{\mathrm{Pic}}

\newcommand{\uG}{\underline{G}}
\font\brus=wncyr10.240pk scaled 1200 .240pk
\begin{document}
\title[Spinor groups]{Spinor groups with good reduction}

\author[V.~Chernousov]{Vladimir I. Chernousov}
\author[A.~Rapinchuk]{Andrei S. Rapinchuk}
\author[I.~Rapinchuk]{Igor A. Rapinchuk}

\address{Department of Mathematics, University of Alberta, Edmonton, Alberta T6G 2G1, Canada}

\email{vladimir@ualberta.ca}

\address{Department of Mathematics, University of Virginia,
Charlottesville, VA 22904-4137, USA}

\email{asr3x@virginia.edu}

\address{Department of Mathematics, Michigan State University, East Lansing, MI
48824, USA}

\email{rapinchu@msu.edu}



\begin{abstract}
Let $K$ be a 2-dimensional global field of characteristic $\neq 2$, and let $V$ be a divisorial set of places of $K$. We show that for a given $n \geqslant 5$, the set of $K$-isomorphism classes of spinor groups $G = \mathrm{Spin}_n(q)$ of nondegenerate $n$-dimensional quadratic forms over $K$ that have good reduction at all $v \in V$, is finite. This result yields some other finiteness properties, such as the finiteness of the genus $\gen_K(G)$ and the properness of the global-to-local map in Galois cohomology. The proof relies on the finiteness of the unramified cohomology groups $H^i(K , \mu_2)_V$ for $i \geqslant 1$ established in the paper. The results for spinor groups are then extended to some unitary groups and to groups of type $\textsf{G}_2$.

\end{abstract}

\dedicatory{{To G\"unter Harder on his 80th birthday}}

\maketitle

\section{Introduction}\label{S:Intro}

Let $K$ be a field. The purpose of this paper is to present new results in the framework of  the following general problem:

\vskip3mm

\noindent (*) \parbox[t]{16cm}{\it (When) can one equip $K$ with a natural set $V$ of discrete valuations such that for a given absolutely almost simple simply connected algebraic $K$-group $G$, the set of $K$-isomorphism classes of (inner) $K$-forms of $G$ having good reduction at all $v \in V$ (resp., at all $v \in V \setminus S$, where $S \subset V$ is an arbitrary finite subset) is {\bf finite}?}

\vskip3mm

\noindent (We refer to the subsection {\it Notations and conventions} below for the definition of good reduction.)

\vskip1mm

While the analysis of abelian varieties defined over a global field and having good reduction at a given set of places of the field has been one of the central topics in arithmetic geometry for a long time, particularly since the work of G.~Faltings \cite{Fal}, similar questions in various situations involving linear algebraic groups have received less attention so far. So, before formulating our results, we would like to include a brief account of the previous work in this direction and discuss connections between (*) and several other topics of active current research, which should  provide the reader with some context for (*) and its variations.

Historically, the consideration of forms of linear algebraic groups with good reduction can be traced back to the work of G.~Harder \cite{Harder}, J.-L.~Colliot-Th\'el\`ene and J.-J.~Sansuc \cite{CTS}, and B.~Gross \cite{Gross}. We note that in \cite{Harder}, the focus was primarily on a number field $K$. In this case, combining well-known results about the properness of the global-to-local map for the Galois cohomology of algebraic groups \cite[Ch. III, 4.6]{Serre-GC} with the fact that an absolutely almost simple algebraic group with good reduction over a $p$-adic field is necessarily quasi-split, one can see that a set $V$ consisting of almost all nonarchimedean places of $K$ satisfies (*) --- see \cite{Gross}, \cite{Conrad} for more precise results over $K = \Q$, and also \cite{JL}.

The case where $K = k(x)$ is the field of rational functions in one variable over a field $k$ and $V$ consists of the discrete valuations $v_{p(x)}$ corresponding to {\it all} irreducible polynomials $p(x) \in k[x]$ was considered by Raghunathan and Ramanathan \cite{RagRam} (see also \cite[Theorem 2.1]{CGP}): their result implies that if $G_0$ is a (connected) semi-simple group over $k$ and $G$ is obtained from $G_0$ by the base change $K/k$, then any $K$-form $G$ that splits over $\bar{k}(x)$ (where $\bar{k}$ is a separable closure of $k$) and has good reduction at all $v \in V$ is obtained by base change from a certain $k$-form of $G_0$.  In the same notations, a description of the $K$-forms of $G$ that (split over $\bar{k}(x)$ and) have good reduction at all $v \in V \setminus \{ v_x \}$ was obtained by Chernousov, Gille and Pianzola \cite{CGP}, which played a crucial role in the proof of the conjugacy of the analogues of Cartan subalgebras in certain infinite-dimensional Lie algebras \cite{CNPY}. (We note that if $k$ has characteristic zero then every semi-simple $K$-group $G$ becomes quasi-split over $\bar{k}(x)$, which implies that those $G$ that have good reduction at all $v \in V \setminus \{ v_x \}$  automatically split over $\bar{k}(x)$.)

In all these instances, $K$ was the fraction field of a certain Dedekind ring and $V$ was the set of discrete valuations associated with the nonzero prime ideals of this ring, making the situation ``1-dimensional." While the finiteness questions have not been fully answered yet even in the 1-dimensional case (cf. Theorem \ref{T:F3} below and subsequent remarks), the current work on the following problems necessitate the analysis of (*) in a more general (higher-dimensional) setting:

\vskip2mm

\noindent $\bullet$ %\parbox[t]{16cm}
{{\it The genus problem.} Given an absolutely almost simple simply connected algebraic $K$-group $G$, its genus $\gen_K(G)$ is defined to be the set of $K$-isomorphism classes of $K$-forms $G'$ of $G$ that have the same isomorphism classes of maximal $K$-tori as $G$ (the latter means that every maximal $K$-torus $T$ of $G$ is $K$-isomorphic to some maximal $K$-torus $T'$ of $G'$, and vice versa). One expects $\gen_K(G)$ to be finite for any $G$ over an arbitrary finitely generated field $K$. This has been proven for groups of all types when $K$ is a number field \cite[Theorem 7.5]{PrRap-WC}, and for inner forms of type $\textsf{A}_{\ell}$ over arbitrary finitely generated fields -- cf. \cite{CRR1} \cite{CRR3}, \cite{CRR5}. The connection between the genus and good reduction described in \cite[Theorem 5]{CRR4}, \cite{CRR6} implies that (under minor additional technical assumptions) the fact that $K$ possesses a set $V$ of discrete valuations such that (*) holds for a given $G$ implies that $\gen_K(G)$ is finite. So, (*) provides a natural approach to proving the finiteness of the genus.}

\vskip4mm

\noindent $\bullet$ %\parbox[t]{16cm}
{{\it The Hasse principle.} For a given algebraic $K$-group $G$ and a given set $V$ of valuations of $K$,
one considers the global-to-local map in Galois cohomology:
$$
\theta_{G , V} \colon H^1(K , G) \longrightarrow \prod_{v \in V} H^1(K_v , G),
$$
(where $K_v$ is the completion of $K$ with respect to $v$), and one says that the Hasse principle holds for $G$ with respect to $V$ if $\theta_{G , V}$ is injective (cf. \cite[4.7]{Serre-GC}). While the Hasse principle does hold in many important situations over a number field $K$ (in which case, one takes for $V$
the set of all valuations of $K$, including archimedean ones) - see \cite[Ch. VI]{Pl-R}, it may fail for general semi-simple groups \cite[4.7]{Serre-GC}, but for {\it any} algebraic group $G$ over a number field $K$, the map $\theta_{G , V}$ is {\it proper} (in other words, the deviation from the Hasse principle is always {\it finite}). It is important to point out that the affirmative answer to (*) would enable one to significantly extend this result, viz. one shows that if (*) holds for a given $K$, $V$ and $G$, then for the corresponding adjoint group $\overline{G}$ the map $\theta_{\overline{G} , V}$ is proper (cf. \S\ref{S:5}.2). See Conjecture B and its discussion below for more concrete statements/results.}

\vskip4mm

\noindent $\bullet$ %\parbox[t]{16cm}
{{\it Eigenvalue rigidity.} Let $G_1$ and $G_2$ be two semi-simple algebraic groups over a field $F$ of characteristic zero, and let $\Gamma_i \subset G_i(F)$ be a Zariski-dense subgroup for $i = 1, 2$. In \cite{PrRap-WC}, the notion of {\it weak commensurability} of $\Gamma_1$ and $\Gamma_2$, based on the consideration of eigenvalues of semi-simple elements of those subgroups, was introduced and developed. A detailed analysis of weakly commensurable arithmetic groups was used in {\it loc. cit.} to derive geometric consequences. It is currently expected that certain results allowing one to (almost) recover some characteristics of an arithmetic group from the information about the eigenvalues of its elements should be valid for arbitrary finitely generated Zariski-dense subgroups -- this phenomenon is called {\it eigenvalue rigidity} \cite{R-ICM}. Central to it is the following {\it Finiteness conjecture} \cite{PrRap-Action}, \cite{R-ICM}: Let $G_1$ and $G_2$ be absolutely simple (adjoint) algebraic groups over a field $F$ of characteristic zero, and let $\Gamma_1 \subset G_1(F)$ be a finitely generated Zariski-dense subgroup with  trace field\footnotemark \ $K = K_{\Gamma_1}$. Then there are only finitely many $F/K$-forms $\mathscr{G}_2^{(1)}, \ldots , \mathscr{G}_2^{(r)}$ of $G_2$ such that any finitely generated Zariski-dense subgroup $\Gamma_2 \subset G_2(F)$ which is weakly commensurable to $\Gamma_1$ is $G_2(F)$-conjugate to a subgroup of $\mathscr{G}_2^{(j)}(K) \subset G_2(F)$ for some $j = 1, \ldots , r$. Note that the field $K$ is finitely generated hence admits
a {\it divisorial} set of places $V$ (see below). Then if (*) holds for $K$, $V$ and an absolutely almost simple simply connected group $G$
of the same type as $G_2$, the above finiteness conjecture is valid, cf. \cite{CRR6}. }

\footnotetext{Defined to be the subfield of $F$ generated by the traces $\mathrm{Tr}\: \mathrm{Ad}_{G_1}(\gamma)$ of all elements $\gamma \in \Gamma_1$ in the adjoint representation.}

\vskip2mm

While various forms of the Hasse principle (including the cohomological form cited above) have been studied extensively for a long time, the genus problem and eigenvalue rigidity emerged relatively recently as an extension of the research carried out in \cite{PrRap-WC} in connection with some geometric problems for isospectral and length-commensurable locally symmetric spaces. So, the fact that (*) has strong consequences for all three problems, combined with the previous work done in the Dedekind case, makes the analysis of (*) in the general case worthwhile.

In order to make (*) more tractable, one can specialize it to a suitable class of fields $K$ or restrict the type of an absolutely almost simple algebraic $K$-group $G$ in its statement (or do both). It is most interesting to investigate (*) for finitely generated fields $K$. In this case, $K$ possesses natural sets of discrete valuations called {\it divisorial}. More precisely, let $\mathfrak{X}$ be a normal irreducible scheme of finite type over $\Z$ (if $\mathrm{char}\: K = 0$) or over a finite field (if $\mathrm{char}\: K > 0$) such that $K$ is the field of rational functions on $\mathfrak{X}$ (we will refer to $\mathfrak{X}$ as a {\it model} of $K$). It is well-known that to every prime divisor $\mathfrak{Z}$ of $\mathfrak{X}$ there corresponds a discrete valuation $v_{\mathfrak{Z}}$ on $K$ (cf. \cite[12.3]{Cut}, \cite[Ch. II, \S 6]{Hart}). Then
$$
V(\mathfrak{X}) = \{ v_{\mathfrak{Z}} \ \vert \ \mathfrak{Z} \ \ \text{prime divisor of} \ \ \mathfrak{X} \}
$$
is called the divisorial set of places of $K$ corresponding to the model $\mathfrak{X}$. Any set of places $V$ of $K$ of this form (for some model $\mathfrak{X}$) will be simply called {\it divisorial}. Note that any two divisorial sets $V_1$ and $V_2$ are commensurable, i.e. $V_i \setminus (V_1 \cap V_2)$ is finite for $i = 1, 2$, and for any finite subset $S$ of a divisorial set $V$ the set $V \setminus S$ contains a divisorial set. We then can formulate the following more precise version of (*).

\vskip2mm

\noindent {\bf Conjecture A.} {\it Let $K$ be a finitely generated field, and $V$ be a divisorial set of places. Then for a given absolutely almost simple algebraic $K$-group $G$, the set of $K$-isomorphism classes of (inner) $K$-forms of $G$ that have good reduction at all $v \in V$ is finite.}

\vskip2mm

M.~Rapoport suggested that a finiteness statement of such nature is likely to be true for arbitrary reductive $K$-groups, viz. the number of $K$-isomorphisms classes of reductive $K$-groups of bounded dimension (or rank) that have good reduction at all $v \in V$ should be finite.
At this point, Conjecture~A is known over arbitrary finitely generated only for inner forms of type $\textsf{A}_{\ell+1}$ with $(\ell + 1)$ prime to $\mathrm{char}\: K$ (cf. \cite{CRR3}, \cite{CRR4}). In fact, before the current paper Conjecture A was not known for any other type over fields other than global.

One of the goals of the paper is to prove Conjecture A for the spinor groups of quadratic forms, and also some unitary groups and groups of type $\textsf{G}_2$, over {\it 2-dimensional global fields} of characteristic $\neq 2$. Following Kato \cite{Kato}, by a 2-dimenional global field, we mean either the function field $K = k(C)$ of a smooth geometrically integral curve $C$ over a number field $k$, or the function field $K = k(S)$ of a smooth geometrically integral surface $S$ over a finite field $k$.
\begin{thm}\label{T:F1}
Let $K$ be a 2-dimensional global field of characteristic $\neq 2$, and let $V$ be a divisorial set of places of $K$. Fix an integer $n \geqslant 5$. Then the set of $K$-isomorphism classes of spinor groups $G = \mathrm{Spin}_n(q)$  of nondegenerate quadratic forms in $n$ variables over $K$ that have good reduction at all $v \in V$ is finite.
\end{thm}

We will now indicate some consequences of this result for two of the three problems discussed above. First, we have the following finiteness result for the genus of spinor groups over 2-dimensional global fields.
\begin{thm}\label{T:genus}
Let $K$ be a 2-dimensional global field of characteristic $\neq 2$, and let $G = \mathrm{Spin}_n(q)$, where $q$ is a $n$-dimensional nondegenerate quadratic form over $K$. If either $n \geqslant 5$ is odd or $n \geqslant 10$ is even and $q$ is $K$-isotropic, then $\gen_K(G)$ is finite.
\end{thm}
(In fact, we prove that the number of $K$-isomorphism classes of spinor groups $G' = \mathrm{Spin}_n(q')$ of nondegenerate $n$-dimensional quadratic forms $q'$ over $K$ that have the same isomorphism classes of maximal $K$-tori as $G$ is finite for any $n \geqslant 5$ --- see Theorem \ref{T:genus1}.)

\vskip1mm

Another application deals with the global-to-local map in Galois cohomology. Namely, the techniques developed to prove Theorem \ref{T:F1} also yield the following.
\begin{thm}\label{T:HP}
Notations as in Theorem \ref{T:F1}, for $G = \mathrm{SO}_n(q)$  the map $$\theta_{G , V} \colon H^1(K , G) \longrightarrow \prod_{v \in V} H^1(K_v , G)$$ is proper, i.e. the pre-image of a finite set is finite.
\end{thm}

Moreover, our techniques in fact yield similar results for some groups $G$ of the form $\mathrm{SL}_{1 , D}$ (cf. Theorem \ref{T:HP1}),
some unitary groups and groups of type $\textsf{G}_2$ (cf. \S\S \ref{S:unitary}-\ref{S:G2}), and a partial result for some spinor groups (cf. Proposition \ref{P:HP2}). Based on these results, we would like to propose the following conjecture.

\vskip2mm

\noindent {\bf Conjecture B.} {\it Let $K$ be a 2-dimensional global field, and $V$ be a divisorial set of places of $K$. Then for any absolutely almost simple (or even semi-simple) algebraic group $G$, the natural map $$\theta_{G , V} \colon H^1(K , G) \longrightarrow \prod_{v \in V} H^1(K_v , G)$$ is proper. In particular, $$\text{\brus Sh}(G , V) := \ker \theta_{G , V}$$
is finite.}

\vskip1mm

(We note that the use of twisting shows that proving the finiteness of $\text{\brus Sh}(G , V)$ is the most essential part of Conjecture B.)

\vskip2mm

Of course, it would be quite tempting to extend this conjecture to all finitely generated fields, but for this more evidence needs to be developed.
At this point, the only result that goes beyond 2-dimensional global fields is  that for any finitely generated field $K$, its divisorial set of places $V$, and any $n$ prime to $\mathrm{char}\: K$, the map  $H^1(K , G) \to \prod_{v \in V} H^1(K_v , G)$ is proper for $G = \mathrm{PGL}_n$, which follows from the finiteness of the unramified Brauer group ${}_n\Br(K)_V$, cf. \cite{CRR3}, \cite{CRR4}.

\vskip2mm

Applications to weakly commensurable Zariski-dense will be given in \cite{CRR6}.

\vskip2mm

Theorem \ref{T:F1} will be derived from a more general Theorem \ref{T:F2} (see \S \ref{S:F2}) that (potentially) enables one to answer (*) for spinor groups over arbitrary finitely generated fields. We will also use this result to prove a finiteness statement for spinor groups, as well as some unitary groups and groups of type $\textsf{G}_2$, with good reduction over a class of fields that are {\it not} finitely generated. This class includes the function fields of $p$-adic curves that have received a great deal of attention in recent years (we refer the reader to  \cite{Brussel}, \cite{CTPS}, \cite{Parim}, \cite{Par-Sur} and  references therein for various results involving division algebras, quadratic forms, and algebraic groups over those fields), but is in fact much larger. We will formulate our results using a generalization of Serre's condition $(\mathrm{F})$ (see \cite[Ch. III, \S 4]{Serre-GC}) offered in \cite{IRap}. Let $K$ be a field and $m \geqslant 1$ an integer prime to $\mathrm{char}\: K$. We then introduce the following condition on $K$:

\medskip

$(\mathrm{F}'_m)$ For every finite separable extension $L/K$, the quotient $L^{\times}/{L^{\times}}^m$ is finite.

\medskip

\noindent (Note that if $L^{\times}/{L^{\times}}^m$ is finite for every finite \emph{separable} extension $L/K$, then it is finite for \emph{any} finite extension of $K$ --- see \cite[Lemma 2.8]{IRap}). Combining Theorem \ref{T:F2} with the results on the finiteness of unramified cohomology with $\mu_m$-coefficients over fields satisfying $(\mathrm{F}'_m)$ \cite{IRap}, we obtain the following.
\begin{thm}\label{T:F3}
Let $C$ be a smooth geometrically integral curve over a field $k$ of characteristic $\neq 2$ that satisfies condition $(\mathrm{F}'_2)$, and let $K = k(C)$ be its function field. Denote by $V$ the set of discrete valuations of $K$ corresponding to the closed points of $C$. Then the number of $K$-isomorphism classes of spinor groups $G = \mathrm{Spin}_n(q)$  of nondegenerate quadratic forms $q$ over $K$ in $n \geqslant 5$ variables that have good reduction at all $v \in V$ is finite.
\end{thm}

\smallskip

Theorem \ref{T:F3} is likely to extend to absolutely almost simple simply connected groups of all types over the function fields of curves defined
over a field that satisfies condition $(\mathrm{F})$ --- see Conjecture 6.3. In this regard, we observe that in \S\S \ref{S:unitary}-\ref{S:G2} we extend the above results for spinor groups to the special unitary groups of hermitian forms over quadratic extensions of the base field and to groups of type $\textsf{G}_2$ (for the same fields $K$ and the same sets of valuations $V$).

The paper is organized as follows. In \S\ref{S:Pic}, we develop some formalism involving the Picard group associated with a set of discrete valuations and then use it to formulate a general result (Theorem \ref{T:F2}) that reduces the proof of the finiteness of the set of isomorphism classes of spinor groups having good reduction to the finiteness of certain unramified cohomology groups. We prove Theorem \ref{T:F2} in \S\ref{S:F2}, where we also give an application of our method to the properness of the global-to-local map in Galois cohomology (Theorem \ref{T:F2-1}). In \S\ref{S:ProofTF1}, we combine Theorem \ref{T:F2} with the finiteness results for unramified cohomology
of 2-dimensional global fields to prove Theorem \ref{T:F1}. It should be pointed out that these finiteness results are known in degrees $\leqslant 2$ and easily follow from the theorems of Poitou-Tate in degrees $\geqslant 4$, but have not appeared in the literature in degree 3. We derive the required fact in degree~3 (Corollary \ref{C:mu2}) from the more general Theorem \ref{T:H3-Finite}. For several reasons (see the discussion in \S\ref{S:ProofTF1}.1), we give two proofs of the characteristic zero case of Theorem \ref{T:H3-Finite}: the one given in \S\ref{S:H3} is based on the referee's suggestions, while the one in the Appendix (\S\ref{S:Append}) is our original argument. The positive characteristic case of Theorem \ref{T:H3-Finite} is treated in \S\ref{S:Fm}, along with the finiteness results involving fields of type $(\mathrm{F}'_m)$ and the proof of Theorem \ref{T:F3}. Theorems \ref{T:genus} and \ref{T:HP} are proved in \S\ref{S:5}. Finally, in \S\S\ref{S:unitary}-\ref{S:G2}, we present finiteness results for special unitary groups and groups of type $\mathsf{G}_2$.

\vskip3mm

\noindent {\bf Notations and conventions.} For a field $k$, we will denote by $\bar{k}$ a fixed separable closure. Given a discrete valuation $v$ of $k$, we let $k_v$ and $k^{(v)}$ denote the completion and the residue field of $k$ at $v$, respectively. We recall that a $\Ga(\bar{k}/k)$-module $M$ is said to be {\it unramified} at $v$ if for some (equivalently, any) extension $w$ of $v$ to $\bar{k}$, the inertia subgroup of the decomposition group ${\rm Dec}_w \subset \Ga(\bar{k}/k)$ acts trivially on $M.$ Also, if $G$ is an absolutely almost simple linear algebraic group defined over $k$, we will say that $G$ has {\it good reduction at $v$} if there exists a reductive group scheme $\mathscr{G}$ over the valuation ring $\mathcal{O}_v$ of $k_v$ whose generic fiber $\mathscr{G} \otimes_{\mathcal{O}_v} k_v$ is isomorphic to $G \otimes_k k_v$ (this definition, involving completions, is convenient for applications to the Hasse principle).

We will follow the conventions outlined in \cite[Ch. II, \S7.8]{GMS} regarding Tate twists of Galois modules. Namely, suppose $v$ is a discrete valuation of $k$, $n$ an integer prime to $\mathrm{char}\: k^{(v)}$, and $M$ a finite $\Ga(\bar{k}/k)$-module satisfying $nM = 0.$ For an integer $d$, one defines $M(d)$ to be $\mu_n^{\otimes d} \otimes M$ if $d \geq 0$ and $\mathrm{Hom}(\mu_n^{\otimes(-d)}, M)$ if $d < 0.$ In particular, $M(-1) = \mathrm{Hom}(\mu_n, M).$ In the case where $M$ is torsion (but not necessarily finite) without any elements of order equal to the residue characteristic, one writes $M = \displaystyle{\lim_{\longrightarrow} M'}$, where $M'$ are the finite submodules of $M$, and sets $M(d) = \displaystyle{\lim_{\longrightarrow} M'(d)}$. As usual, we will use $\mu_n^{\otimes d}$ to denote $\Z/n\Z(d)$ for {\it all} $d$.

Finally, we recall that if $n$ is an integer prime to ${\rm char}~k$, then isomorphism $k^{\times}/{k^{\times}}^n \simeq H^1(k, \mu_n)$ from Kummer theory is induced by sending an element $a \in k^{\times}$ to the cohomology class of the 1-cocycle
$$
\chi_{n,a} (\sigma) = \sigma(\sqrt[n]{a})/\sqrt[n]{a} \ \ \ \text{for} \ \sigma \in \Ga(\bar{k}/k).
$$
When $n = 2$, we will denote this cocycle simply by $\chi_a.$

\section{The Picard group associated with a set of discrete valuations}\label{S:Pic}

Suppose that a field $K$ is equipped with a set $V$ of discrete valuations that satisfies the following condition

\medskip

$(\mathrm{A})$  For any $a \in K^{\times}$, the set \: $V(a) := \{ v \in V \ \vert \ v(a) \neq 0 \}$ \: is finite.

\medskip

\noindent (It is worth noting that $(\mathrm{A})$ automatically holds for a divisorial set of valuations $V$ of a finitely generated field $K$.)
We now let $\mathrm{Div}(V)$ denote the free abelian group on the set $V$, the elements of which will be called ``{\it divisors}." The fact that $V$ satisfies $(\mathrm{A})$ enables one to associate to any $a \in K^{\times}$  the corresponding ``{\it principal divisor}"
$$
(a) = \sum_{v \in V} v(a) \cdot v.
$$
 Let $\mathrm{P}(V)$ denote the subgroup of $\mathrm{Div}(V)$ formed by all principal divisors. We call the quotient $\mathrm{Div}(V)/\mathrm{P}(V)$ the {\it Picard group} of $V$ and denote it by $\mathrm{Pic}(V)$.

Next, we recall (cf. \cite[Ch. II]{GMS}) that for a discrete valuation $v$ of a field $K$ and a finite Galois module $M$ which is unramified at $v$ and whose order is prime to $\mathrm{char}\: K^{(v)}$, one has a residue map
$$
\partial^M_{i , v} \colon H^{i}(K , M) \longrightarrow H^{i-1}(K^{(v)} , M(-1)) \ \ \ (i \geqslant 1).
$$
In particular, for every $n$ prime to $\mathrm{char}\: K^{(v)}$ and every $d$ we have the residue map
$$
\partial^{n , d}_{i , v} \colon H^{i}(K , \mu_n^{\otimes d}) \longrightarrow H^{i-1}(K^{(v)} , \mu_n^{\otimes (d-1)}).
$$
An element of $H^{i}(K , M)$ (in particular, of $H^{i}(K , \mu_n^{\otimes d})$) is said to be {\it unramified} if it lies in the kernel of the relevant residue map.

%its image under the relevant residue map is trivial.

In this section, we will only consider cohomology with coefficients in $\mu_2 = \{ \pm 1 \}$. Then for a discrete valuation $v$ of $K$, the corresponding residue map is defined whenever $\mathrm{char}\: K^{(v)} \neq 2$ and will be denoted by
$$
\partial^i_v \colon H^i(K , \mu_2) \longrightarrow H^{i-1}(K^{(v)}, \mu_2).
$$
We now make the following assumption

\medskip

$(\mathrm{B})$ $\mathrm{char}\: K^{(v)} \neq 2$ for all $v \in V$.

\medskip

\noindent We then define the $i$-th unramified cohomology group of $K$ with respect to $V$ by
%groups of unramified cohomology with respect to $V$ by
$$
H^i(K , \mu_2)_V = \bigcap_{v \in V} \mathrm{Ker}\: \partial^i_v.
$$
With these notations, we have

\begin{thm}\label{T:F2}
Let $K$ be a field equipped with a set $V$ of discrete valuations satisfying conditions $(\mathrm{A})$ and $(\mathrm{B})$, and let $n \geqslant 5$ be an integer. Assume that

\smallskip

{\rm (1)} the quotient $\Pic(V)/2 \cdot \Pic(V)$ is finite; and

\smallskip

{\rm (2)} the unramified cohomology groups $H^i(K , \mu_2)_{V}$ are finite for all $i = 1, \ldots , \ell := [\log_2 n] + 1$.

\smallskip

\noindent Then the number of $K$-isomorphism classes of spinor groups $G = \mathrm{Spin}_n(q)$  of nondegenerate quadratic forms $q$ over $K$ in $n$ variables that have good reduction at all $v \in V$ is  $$\displaystyle \leqslant | \Pic(V)/2 \cdot \Pic(V) | \cdot \prod_{i = 1}^{\ell} | H^i(K , \mu_2)_{V} |$$ (in particular, finite).
\end{thm}

\vskip1mm

We will postpone the proof of Theorem \ref{T:F2} until the next section, and recall now the connection between $\Pic(V)$ and the ideles, which is well-known in the classical setting (cf. \cite[Ch. II, \S 17]{ANT}). Given a field $K$ endowed with a set $V$ of discrete valuations that satisfies (A), we define the \emph{group of ideles} $\mathbb{I}(K , V)$ as the restricted direct product of the multiplicative groups $K^{\times}_v$ for $v \in V$ with respect to the groups of units $U_v = \mathcal{O}^{\times}_v$:
$$
\mathbb{I}(K , V) = \{\, (x_v) \in \prod_{v \in V} K^{\times}_v \ \, \vert \ \, x_v \in U_v \ \ \text{for almost all} \ \ v \in V \,\}.
$$
Furthermore, we let
$$
\mathbb{I}_0(K , V) = \prod_{v \in V} U_v
$$
be the \emph{subgroup of integral ideles}. Since $V$ satisfies (A), one can consider the diagonal embedding $K^{\times} \hookrightarrow \mathbb{I}(K , V)$, the image of which will be called the \emph{group of principal ideles} and denoted also by $K^{\times}$.  Then
$$
\nu \colon \mathbb{I}(K , V) \to \mathrm{Div}(V), \ \ (x_v) \mapsto \sum_{v \in V} v(x_v) \cdot v
$$
is a surjective group homomorphism with kernel $\ker \nu = \mathbb{I}_0(K , V)$. Thus, we obtain the following.
\begin{lemma}\label{L:Id}
The map $\nu$ induces a natural identification of the quotient $\mathbb{I}(K , V) / \mathbb{I}_0(K , V) K^{\times}$ with $\mathrm{Pic}(V)$.
\end{lemma}

\bigskip

\section{Proof of Theorem \ref{T:F2} and its variations}\label{S:F2}

\noindent {\bf 1. Two facts about the Witt ring.} Let $F$ be a field of characteristic $\neq 2$. We let $W(F)$ and $I(F)$ denote the Witt ring of $F$  and its fundamental ideal, respectively. For a nondegenerate quadratic form $q$ over $F$,  $[q]$ denotes the corresponding class in $W(F)$. As usual, the quadratic form $a_1x_1^2 + \cdots + a_nx_n^2$ with $a_i \in F^{\times}$ will be denoted by $\langle a_1, \ldots , a_n \rangle$, while $\langle\!\langle a_1, \ldots a_d   \rangle\!\rangle$ will be used to denote the $d$-fold Pfister form $\langle 1 , -a_1 \rangle \otimes \cdots \otimes \langle 1 , -a_d \rangle$. Clearly, for any $d \geqslant 1$, the $d$-th power $I(F)^d$ is additively generated by the classes of $d$-fold Pfister forms.
\begin{lemma}\label{L:easy}
Let $q$ be a nondegenerate quadratic form over $F$ such that $[q] \in I(F)^d$ for some $d \geqslant 1$. Then for any $\lambda \in F^{\times}$, we have $[\lambda q] \in I(F)^d$ and  $[\lambda q] + I(F)^{d+1} = [q] + I(F)^{d + 1}$.
\end{lemma}
\begin{proof}
The fact that $[\lambda q] \in I(F)^d$ is obvious. Furthermore, we have
$$
[q] - [\lambda q] = [ q \perp (-\lambda q)] = [<1 , -\lambda> \otimes q] \in I(F)^{d+1},
$$
as required.
\end{proof}

Now, let $F$ be a field complete with respect to a discrete valuation $v$ such that the characteristic of the residue field $F^{(v)}$ is $\neq 2$. We let $U(F)$ denote the group of units in $F$, and fix a uniformizer $\pi$.
Then one can define the first and second residue maps
$$
\partial_i \colon W(F) \to W(F^{(v)}), \ \ i = 1, 2,
$$
which are homomorphisms of additive groups uniquely characterized by the conditions
$$
\partial_1\langle u \rangle = \langle \bar{u} \rangle, \ \ \partial_1\langle\pi u\rangle = 0, \ \ \ \text{and} \ \ \ \partial_2\langle u \rangle = 0, \ \
\partial_2\langle\pi u \rangle = \langle \bar{u} \rangle,
$$
for any $u \in U(F)$ (where $\bar{u}$ denote the image of $u$ in $F^{(v)}$) (see \cite[Ch. VI, \S 1]{Lam} or \cite[\S 5]{Milnor} for the details). Let $W_0(F)$ be the subring of $W(F)$ generated by the classes of $\langle u \rangle$ for $u \in U(F)$. Then $W_0(F) = \mathrm{Ker}\: \partial_2$, while $\partial_1$ yields an isomorphism between $W_0(F)$ and $W(F^{(v)})$. Thus, we have the following split exact sequence
$$
0 \to W_0(F) \longrightarrow W(F) \stackrel{\partial_2}{\longrightarrow} W(F^{(v)}) \to 0.
$$
Let $I_0(F) = W_0(F) \cap I(F)$. Milnor  \cite[\S 5]{Milnor} shows that for any $d \geqslant 1$, the restriction of $\partial_2$ to $I^d(F)$
(the $d$th power of $I(F)$) yields the exact sequence
$$
0 \to I_0^d(F) \longrightarrow I^d(F) \stackrel{\partial_2}{\longrightarrow} I^{d-1}(F^{(v)}) \to 0.
$$
This, in particular, gives the following
\begin{lemma}\label{L:Witt1}
For any $d \geqslant 1$, we have $I(F)^d \cap W_0(F) = I_0(F)^d$.
\end{lemma}

\bigskip

\noindent {\bf 2. The Milnor isomorphism and unramified classes.} Let again $F$ be a field of characteristic $\neq 2$. It is a consequence of Voevodsky's proof of the Milnor conjecture (see \cite{OVV}, \cite{V1}, \cite{V2}) that for $d \geqslant 1$, there are natural isomorphisms of abelian groups
$$
\gamma_{F , d} \colon I(F)^d / I(F)^{d + 1} \longrightarrow H^d(F , \mu_2).
%\ \ \ \text{(where}\ \  \mu_2 = \{ \pm 1 \}).
$$
Explicitly, $\gamma_{F,d}$ is defined by sending the class of the Pfister form $\langle\!\langle a_1, \ldots , a_d \rangle\!\rangle$ to the cup-product $\chi_{a_1} \cup \cdots \cup \chi_{a_d}$, where  for $a \in F^{\times}$, we let $\chi_a$ be the corresponding 1-cocycle given by Kummer theory.
%we let $\chi_F(a)$ denote the corresponding character/1-cocycle $\Ga(\overline{F}/F) \to \mu_2$ defined by
%$$
%\chi_F(a)(\sigma) = \sigma(\sqrt{a})/\sqrt{a} \ \ \text{for} \ \ \sigma \in \Ga(\overline{F}/F).
%$$
\begin{lemma}\label{L:unram111}
Let $F$ be a field complete with respect to a discrete valuation $v$ such that $\mathrm{char}\: F^{(v)} \neq 2$, and let $d \geqslant 1$. If $q$ is a nondegenerate quadratic form over $F$ such that $[\lambda q] \in I(F)^d \cap W_0(F)$ for some $\lambda \in F^{\times}$ (notations as in the previous subsection), then $[q] \in I(F)^d$ and the cohomology class
$\gamma_{F , d}([q]) \in H^d(F , \mu_2)$ is unramified at $v$.
\end{lemma}
\begin{proof}
It immediately follows from Lemma \ref{L:easy} that $[q] \in I(F)^d$ and $\gamma_{F , d}([q]) = \gamma_{F , d}([\lambda q])$. So, it is enough to show that  if $[q] \in I(F)^d \cap W_0(F)$, then $\gamma_{F , d}([q])$ is unramified at $v$. According to Lemma~\ref{L:Witt1}, we have $I(F)^d \cap W_0(F) = I_0(F)^d$, which is additively generated by the classes of Pfister forms $\langle\!\langle a_1, \ldots , a_d \rangle\!\rangle$ with $a_i \in U(F)$. But for any such form,  the corresponding class
$$
\gamma_{F , d}([\langle\!\langle a_1, \ldots , a_d \rangle\!\rangle]) = \chi_{a_1} \cup \cdots \cup \chi_{a_d}
$$
is clearly unramified, and our claim follows.
\end{proof}

\bigskip

\noindent {\bf 3. Proof of Theorem \ref{T:F2}.} Let $\{ q_i \}_{i \in I}$ be a family of $n$-dimensional nondegenerate quadratic forms over $K$ such that

\vskip2mm

\noindent $\bullet$ for each $i \in I$, the spinor group $G_i = \mathrm{Spin}_n(q_i)$ has  good reduction at all $v \in V$; and

\vskip1mm

\noindent $\bullet$ \parbox[t]{15cm}{for $i , j \in I$, $i \neq j$, the forms $q_i$ and $q_j$ are not similar (i.e., $q_i$ is not equivalent to any nonzero scalar multiple of $q_j$).}

\vskip2mm

\noindent We wish to show that $I$ is finite and
\begin{equation}\label{E:estimate1}
| I | \leqslant d_0 d_1 \cdots d_{\ell}
\end{equation}
where
$$
d_0 = | \Pic(V) / 2 \cdot \Pic(V) | \ \ \ \text{and} \ \ \ d_i = | H^i(K , \mu_2)_V | \ \ \text{for} \ \ i = 1, \ldots , \ell = [\log_2 n] + 1.
$$
We begin by observing that the first of the above conditions implies that for each $i \in I$ and any $v \in V$, there exists $\lambda^{(i)}_v \in K^{\times}_v$ such that the form $\lambda^{(i)}_v q_i \in W_0(K_v)$ in the above notations. In addition, we may assume without loss of generality that for each $i \in I$, we have $\lambda^{(i)}_v = 1$ for almost all $v$; then $\lambda^{(i)} := (\lambda^{(i)}_v)_{v \in V} \in \mathbb{I}(K , V)$. Using Lemma \ref{L:Id}, we see that
$$
\I(K , V)/\I(K , V)^2 \I_0(K , V) K^{\times} \simeq \mathrm{Pic}(V) / 2 \cdot \mathrm{Pic}(V).
$$
So, there exists a subset $J_0 \subset I$ of size $\geqslant \vert I \vert / d_0$ (if $I$ is infinite then this simply means that $J_0$ is also infinite) such that all $\lambda^{(i)}$'s for $i \in J_0$ have the same image in $\I(K , V)/\I(K , V)^2 \I_0(K , V) K^{\times}$. Now, fix $j_0 \in J_0$. For  any $j \in J_0$, we can write
$$
\lambda^{(j)} = \lambda^{(j_0)} (\alpha^{(j)})^2 \beta^{(j)} \delta^{(j)} \ \ \ \text{with} \ \ \ \alpha^{(j)} \in \I(K , V), \ \ \beta^{(j)} \in \I_0(K , V), \ \ \delta^{(j)} \in K^{\times},
$$
(assuming that the elements $\alpha^{(j_0)}$, $\beta^{(j_0)}$ and $\delta^{(j_0)}$ are all trivial), and then set
$$
\tilde{q}_j = \delta^{(j)} q_j \ \ \text{and} \ \ \tilde{\lambda}^{(j)} = \lambda^{(j)} (\delta^{(j)})^{-1}.
$$
We note that the forms $\tilde{q}_j$ for $j \in J_0$ remain pairwise non-similar, in particular, inequivalent. Furthermore, for any $v \in V$ we have $\tilde{\lambda}^{(j)}_v \tilde{q}_j = \lambda^{(j)}_v q_j$, hence $[\tilde{\lambda}^{(j)}_v \tilde{q}_j] \in W_0(K_v)$, and
$$
\tilde{\lambda}^{(j)} = \tilde{\lambda}^{(j_0)} (\alpha^{(j)})^2 \beta^{(j)}.
$$
Then for any $v \in V$, the form
$$
q(j , v) := \tilde{\lambda}^{(j_0)}_v(\tilde{q}_j \perp (-\tilde{q}_{j_0})) = (\alpha^{(j)}_v)^{-2}(\beta^{(j)})^{-1} (\tilde{\lambda}^{(j)}_v\tilde{q}_j)
\perp \tilde{\lambda}^{(j_0)}_v (-\tilde{q}_{j_0})
$$
is equivalent to
$$
(\beta^{(j)}_v)^{-1}(\tilde{\lambda}^{(j)}_v\tilde{q}_j) \perp \tilde{\lambda}^{(j_0)}_v (-\tilde{q}_{j_0}).
$$
Since $\beta^{(j)}_v \in U(K_v)$, we see that $[q(j , v)] \in W_0(K_v) \cap I(K_v)$. Now, invoking Lemma \ref{L:unram111} with $d = 1$, we obtain that
$\gamma_{K_v , 1}([\tilde{q}_j] - [\tilde{q}_{j_0}]) \in H^1(K_v , \mu_2)$ is unramified at $v$. This being true for all $v \in V$, we conclude that
$$
\gamma_{K , 1}([\tilde{q}_j] - [\tilde{q}_{j_0}]) \in H^1(K , \mu_2)_V.
$$
Then one can find a subset $J_1 \subset J_0$ of size $$\geqslant | J_0 |/d_1 \geqslant | I | / (d_0 d_1)$$ (again, if $I$ is infinite this simply means that $J_1$ is also infinite) such that the elements $[\tilde{q}_j] - [\tilde{q}_{j_0}] \in I(K)$ for $j \in J_1$ have the same image under $\gamma_{K , 1}$. Fix $j_1 \in J_1$. Then for any $j \in J_1$ we have
$$
\gamma_{K , 1}([\tilde{q}_j] - [\tilde{q}_{j_1}]) = \gamma_{K , 1}(([\tilde{q}_j] - [\tilde{q}_{j_0}]) - ([\tilde{q}_{j_1}] - [\tilde{q}_{j_0}])) = 0,
$$
implying that $[\tilde{q}_j] - [\tilde{q}_{j_1}] \in I(K)^2$. Furthermore, we observe that
$$
\tilde{\lambda}^{(j)} = \tilde{\lambda}^{(j_1)} (\bar{\alpha}^{(j)})^2 \bar{\beta}^{(j)} \ \ \text{with} \ \ \bar{\alpha}^{(j)} \in \I(K , V), \ \bar{\beta}^{(j)} \in \I_0(K , V),
$$
using which one shows that for each $v \in V$ the class of the form $\tilde{\lambda}^{(j_1)}_v(\tilde{q}_j \perp (-\tilde{q}_{j_1}))$ lies in $W_0(K_v) \cap I(K_v)^2$. Now, Lemma \ref{L:unram111} with $d = 2$ yields that
$$
\gamma_{K , 2}([\tilde{q}_j] - [\tilde{q}_{j_1}]) \in H^2(K , \mu_2)_V.
$$
Then there exists a subset $J_2 \subset J_1$ of size $\geqslant | I | /(d_0d_1d_2)$ such that the elements $[\tilde{q}_j] - [\tilde{q}_{j_1}]$ for $j \in J_2$ have the same image under $\gamma_{K , 2}$. Consequently,  fixing $j_2 \in J_2$, we will have $[\tilde{q}_{j}] - [\tilde{q}_{j_2}] \in I(K)^3$ for all $j \in J_2$. Proceeding inductively, we construct nested sequence of subsets
$$
I \supset J_0 \supset J_1 \supset J_2 \supset \cdots \supset J_{\ell}
$$
such that

\vskip2mm

\noindent (a) $\displaystyle | J_m | \geqslant \frac{| I |}{d_0 d_1 \cdots d_m}$; and

\vskip1mm

\noindent (b) fixing $j_m \in J_m$, we will have $[\tilde{q}_j] - [\tilde{q}_{j_m}] \in I(K)^{m+1}$ for all $j \in J_m$,

\vskip2mm

\noindent for any $m = 1, \ldots , \ell$. However, according to \cite[Ch. X, Hauptsatz 5.1]{Lam}, the dimension of any positive-dimensional anisotropic form in $I(K)^{\ell + 1}$ is $\geqslant 2^{\ell+1} > 2^{\log_2 n + 1} = 2n$. So, the fact that
$$
[\tilde{q}_j] - [\tilde{q}_{j_{\ell}}] = [\tilde{q}_j \perp (- \tilde{q}_{j_{\ell}})] \in I(K)^{\ell + 1}
$$
implies that the form $\tilde{q}_j \perp (- \tilde{q}_{j_{\ell}})$ is hyperbolic, and consequently $\tilde{q}_j$ and $\tilde{q}_{j_{\ell}}$ are equivalent. Since $j \in J_{\ell}$ was arbitrary and the forms $\tilde{q}_j$ for $j \in J_{\ell}$ are pairwise inequivalent, we see that $J_{\ell}$ actually reduces to a single element. Then the inequality in (a) yields the required estimation (\ref{E:estimate1}). \hfill $\Box$

\bigskip

\noindent {\bf 4. Another application of the method.} The method developed to prove Theorem \ref{T:F2} can be used in various situations. Here we would like to indicate one application to the analysis of the global-to-local map in Galois cohomology.
\begin{thm}\label{T:F2-1}
Let $K$ be a field equipped with a set $V$ of discrete valuations satisfying conditions $(\mathrm{A})$ and $(\mathrm{B})$, and let $n \geqslant 5$ be an integer. Assume that

\vskip2mm

\noindent {$(\bullet)$} \parbox[t]{15.7cm}{for each $i = 1, \ldots , [\log_2 n] + 1=\ell$, the kernel $\Omega_i$ of the diagonal map $$\displaystyle H^i(K , \mu_2) \longrightarrow \prod_{v \in V} H^i(K_v , \mu_2)$$ is finite of order $\omega_i$.}

\vskip2mm

\noindent Then for a nondegenerate $n$-dimensional quadratic form $q$ over $K$ and the diagonal map $$\displaystyle \pi \colon H^1(K , \mathrm{SO}_n(q)) \longrightarrow \prod_{v \in V} H^1(K_v , \mathrm{SO}_n(q)),$$ the set $\pi^{-1}(\pi(h))$ is finite of size $\leqslant \omega_1 \cdots \omega_{\ell}$ for any $h \in H^1(K , \mathrm{SO}_n(q))$. In particular, $\pi$ is a proper map.
\end{thm}
\begin{proof}
Let $\pi^{-1}(\pi(h)) = \{ h_i  \}_{i \in I}$, and let $q_i$ be the quadratic form obtained from $q$ by twisting using $h_i$. It is well-known that the $K$-equivalence classes of nondegenerate $n$-dimensional quadratic forms over $K$ are in a natural bijective correspondence with the elements of $H^1(K , \mathrm{O}_n(q))$. Since the map $H^1(K , \mathrm{SO}_n(q)) \to H^1(K , \mathrm{O}_n(q))$ is injective (cf. \cite[\S 6.6]{Pl-R}), the forms $q_i$ for $i \in I$ are pairwise inequivalent. On the other hand, by our construction, the form $q_i$ for any $i \in I$ is equivalent over $K_v$ to the form $q_0$ obtained from $q$ by twisting using $h$, for any $v \in V$. Thus, for any $i , j \in I$, the class $[q_i] - [q_j]$  is trivial in $W(K_v)$ for any $v \in V$. So, using the compatibility of the map $\gamma_{F, d}$ with base change, we conclude that if $[q_i] - [q_j] \in I(K)^d$ then $\gamma_{K , d}([q_i] - [q_j]) \in \Omega_d$.

We now proceed as in the proof of Theorem \ref{T:F2}. Fix $j_0 \in I$. Then there exists a subset $J_1 \subset I$ of size $\geqslant \vert I \vert / \omega_1$ such that all elements $[q_j] - [q_{j_0}] \in I(K)$ for $j \in J_1$ have the same image under $\gamma_{K , 1}$. Fix $j_1 \in J_1$.
Then for any $j \in J_1$ we have
$$
\gamma_{K , 1}([q_j] - [q_{j_1}]) = \gamma_{K , 1}(([q_j] - [q_{j_0}]) - ([q_{j_1}] - [q_{j_0}])) = 0,
$$
which means that $[q_{j}] - [q_{j_1}] \in I(K)^2$. Inductively, we construct a sequence of subsets $I \supset J_1 \supset J_2 \supset \cdots \supset J_{\ell}$ such that

\vskip2mm

\noindent $(\mathrm{a}')$ $\displaystyle | J_m | \geqslant \frac{| I |}{\omega_1 \cdots \omega_m}$;

\vskip1mm

\noindent $(\mathrm{b}')$ fixing $j_m \in J_m$, we will have $[q_j] - [q_{j_m}] \in I(K)^{m+1}$ for all $j \in J_m$,

\vskip2mm

\noindent for any $m = 1, \ldots , \ell$. As in the proof of Theorem \ref{T:F2}, we see that the fact that $[q_j] - [q_{j_{\ell}}] \in I(K)^{\ell + 1}$ implies that $q_j$ and $q_{j_{\ell}}$ are actually equivalent. This means that $J_{\ell}$ reduces to a single element, and then the inequality in  $(\mathrm{a}')$ yields the required estimation.
\end{proof}

We observe that for any $i$, we have the inclusion $\Omega_i \subset H^i(K , \mu_2)_V$ in the above notations. In particular, the finiteness of $H^i(K , \mu_2)_V$ implies that of $\Omega_i$, which enables us to apply Theorem \ref{T:F2-1}.

\vskip5mm

\section{Proof of Theorem \ref{T:F1}}\label{S:ProofTF1}

%\vskip1cm

\noindent {\bf 1. The proof.} Let $K$ be a 2-dimensional global field of characteristic $\neq 2$, and let $V$ be any divisorial set of places of $K$ associated with a (normal) model $\fX$ of finite type on which $2$ is invertible (thus, $\mathrm{char}\: K^{(v)} \neq 2$ for all $v \in V$). According to  Theorem \ref{T:F2}, it is enough to establish the finiteness of the quotient $\mathrm{Pic}(V)/2 \cdot \mathrm{Pic}(V)$ and of the unramified cohomology groups $H^i(K , \mu_2)_V$ for all $i \geqslant 1$. But it easily follows from the definitions that the group $\Pic(V)$ coincides with the usual Picard group $\Pic(\fX)$ of the scheme $\fX$. So, the finiteness of $\Pic(V) / 2 \cdot \Pic(V)$ is a consequence of the following well-known statement.

\begin{prop}\label{P:Pic}
Let $\fX$ be an irreducible normal scheme of finite type over $\Z$ or a finite field, and let $n \geqslant 2$ be an integer which is invertible on $\fX$. Then the quotient $\Pic(\fX) / n \cdot \Pic(\fX)$ and the $n$-torsion ${}_n\Pic(\fX)$ are finite groups.
\end{prop}

\noindent {\it Sketch of proof.} Since $n$ is invertible on $\fX$, we have the following Kummer sequence of \'etale sheaves on $\fX$
%the exact sequence of group schemes
$$
1 \to \mu_{n, \fX} \longrightarrow \mathbb{G}_{m, \fX} \stackrel{[n]}{\longrightarrow} \mathbb{G}_{m, \fX} \to 1,
$$
which gives rise to the long exact sequence
%gives rise to an exact sequence of the corresponding etale sheaves on $\fX$ denoted by adding $\fX$ in the subscript (``Kummer sequence"). The corresponding long exact sequence of etale cohomology looks as follows
$$
H^1(\fX , \mu_{n , \fX}) \longrightarrow H^1(\fX , \mathbb{G}_{m , \fX}) \stackrel{[n]}{\longrightarrow} H^1(\fX , \mathbb{G}_{m , \fX}) \longrightarrow H^2(\fX , \mu_{n , \fX})
$$
in \'etale cohomology.
Since $\Pic(\fX) = H^1(\fX , \mathbb{G}_{m , \fX})$ (cf. \cite[Ch. III, Proposition 4.9]{Milne}), we have a surjection $H^1(\fX , \mu_{n , \fX}) \twoheadrightarrow {}_n\Pic(\fX)$ and an injection $\Pic(\fX) / n \cdot \Pic(\fX) \hookrightarrow H^2(\fX , \mu_{n , \fX})$. So, it is enough to prove the finiteness of $H^i(\fX , \mu_{n , \fX})$ for $i = 1, 2$. In fact, these groups are finite for {\it all} $i$. For $\fX$ of finite type over a finite field, this is a consequence of \cite[Th\'eor\`eme 1.1 in ``Th\'eor\`emes de finitude"]{Del} and the Leray spectral sequence \cite[Ch. III, Theorem 1.18]{Milne} in general and is \cite[Ch. VI, Corollary 5.5]{Milne} for $\fX$ smooth. For $\fX$ of finite type over $\Z$, this is \cite[Theorem 10.2]{CRR3}.
\hfill $\Box$

\vskip2mm

It should be noted that B.~Kahn \cite{Kahn} in fact showed that the group $\Pic(\fX)$ is finitely generated.

\vskip2mm

Combining the finiteness of ${}_2\Pic(V)$ with the fact that the group of $V$-units $$U(V) = \{\, a \in K^{\times} \: \vert \: v(a) = 0 \ \ \text{for all} \ \ v \in V\,\}$$ is finitely generated (cf. \cite{Samuel}) and applying \cite[Proposition 5.1(a)]{CRR3}, we obtain that $H^1(K , \mu_2)_V$ is finite. The finiteness of $H^2(K , \mu_2)_{V} = {}_2\Br(K)_V$ was established in \cite[\S 10]{CRR3}. The finiteness of $H^3(K , \mu_2)_V$ is a new result (see Corollary \ref{C:mu2} of the more general  Theorem \ref{T:H3-Finite}) whose proof requires separate arguments in the characteristic zero and positive characteristic cases. The proof for the characteristic zero case that is given in \S \ref{S:H3}  was suggested by one of the referees; it uses several powerful results, first and foremost, those of Kato \cite{Kato} on cohomological Hasse principles. We have included our original proof of Theorem \ref{T:H3-Finite} in the Appendix (\S\ref{S:Append}). It uses considerably less input; in particular, it does not rely on Kato's local-global principle but is based instead  on a modification of Jannsen's \cite{Jann} proof of the latter. The reason we decided to keep this argument in the paper is that it appears to be more amenable to generalizations in the spirit of Jannsen's proof \cite{Jann2} of Kato's local-global principle for higher-dimensional varieties that generalized his original argument \cite{Jann}.
The finiteness of $H^3(K , \mu_2)_V$ for 2-dimensional global fields of positive characteristic will be established in \S \ref{S:Fm} (we note that this result is implicitly contained in \cite[Proposition 4.4 and Remark 4.5]{IRap}).

Since the cohomological dimension of a 2-dimensional global field of positive characteristic is $\leqslant 3$ (cf. \cite[Ch. II, 4.2]{Serre-GC}), the groups $H^i(K , \mu_2)$ vanish for all $i \geqslant 4$. We will now show that in characteristic zero, the finiteness of $H^i(K , \mu_2)_V$ for $i \geqslant 4$ easily follows from the results of Poitou-Tate (cf. \cite[Ch. II, 6.3]{Serre-GC}), which will complete the proof of Theorem \ref{T:F1}.
\begin{prop}\label{P:PT}
Let $C$ be a smooth (but not necessarily projective) geometrically integral curve over a number field $k$, and $K = k(C)$. Then for $i \geqslant 4$, the groups $H^i(K , \mu_2)_{V_0}$, where $V_0$ is the set of geometric places of $K$ associated with the closed points of $C$, are finite.
\end{prop}

To apply the proposition in our situation, we observe that given a divisorial set of places $V$ of a 2-dimensional global field $K$ of characteristic zero, there exists a smooth geometrically connected curve $C$ defined over a number field $k$ such that $K = k(C)$ and for the set of geometric places $V_0$ of $K$ associated with the closed points of $C$, we have the inclusion
$$
H^i(K , \mu_2)_V \, \subset \, H^i(K , \mu_2)_{V_0}.
$$
For the proof of the proposition,  we first need to review the exact sequence for the \'etale cohomology of a curve, which will also be used in the Appendix.

\vskip1mm

\noindent {\bf 2. The fundamental exact sequence.} Let $C$ be a geometrically integral smooth affine curve over an arbitrary field $k$, and let $n \geqslant 2$ be an integer prime to $\mathrm{char}\: k$. Then for any integer $d$, we consider the Hochschild-Serre spectral sequence in \'etale cohomology
$$
E_2^{p,q} = H^p(k , H^q(C \otimes_k \bar{k} , \mu_n^{\otimes d})) \, \Rightarrow H^{p+q}(C , \mu_n^{\otimes d}).
$$
%(where $\bar{k}$ is a fixed separable closure of $k$).
Since $C$ is affine, we have $H^p(C \otimes_k \bar{k} , \mu_n^{\otimes d}) = 0$ for $p \geqslant 2$ (cf. \cite[Lemma 65.3]{EC-Stack}, \cite[Ch. V, \S 2]{Milne}). So, for each $\ell \geqslant 1$, the spectral sequence yields the following short exact sequence:
\begin{equation}\label{E:X0}
 H^{\ell}(k , \mu_n^{\otimes d}) \stackrel{\iota^{d , \ell}_{k , n}}{\longrightarrow} H^{\ell}(C , \mu_n^{\otimes d}) \stackrel{\omega^{d ,\ell}_{k , n}}{\longrightarrow} H^{\ell-1}(k , H^1(C \otimes_k \bar{k} , \mu_n^{\otimes d})).
\end{equation}
Furthermore, let $V_0$ denote the set of places of $K = k(C)$ associated with the closed points of $C$. Then, since $n$ is prime to $\mathrm{char}\: k$,  it follows from the Bloch-Ogus spectral sequence (see \cite{Bl-Og}, \cite{CT-H-K}, \cite{JannSS}) that for each $\ell \geqslant 1$, the natural map $H^{\ell}(C , \mu_n^{\otimes d}) \to H^{\ell}(K , \mu_n^{\otimes d})_{V_0}$ is surjective. (This is a consequence of the fact that the $E_2$-terms in the Bloch-Ogus spectral sequence satisfy
%It is a consequence of the fact that in the sequence
$E_2^{p , q} = 0$ for all $p > \dim C = 1$ and all $q$, while $E_2^{0 , q} = H^q(K , \mu_n^{\otimes d})_{V_0}$.)
%under the assumption that
%(assuming that $n$ is prime to $\mathrm{char}\: k$).)

\vskip5mm

Let
$$
M(n , d) = H^1(C \otimes_k \bar{k} , \mu_n^{\otimes d}).
$$
Equivalently,
\begin{equation}\label{E-FundGp}
M(n , d) = \mathrm{Hom}(\pi_1(C \otimes_k \bar{k} , \bar{\eta}) , \mu_n^{\otimes d}) \ \ \text{where} \ \ \bar{\eta} = \mathrm{Spec}\: \bar{K}.
\end{equation}
It follows that we have the following equality for the Tate twist
$$
M(n , d)(\ell) = M(n , d + \ell).
$$

\vskip2mm

\noindent {\bf 3. Proof of Proposition \ref{P:PT}.} We may assume that $C$ is affine, and it is enough to show that $H^i(C , \mu_n^{\otimes d})$ is finite for all $i \geqslant 4$. According to Poitou-Tate, for a number field $k$ and any finite Galois module $M$, the natural homomorphism
$$
H^i(k , M) \longrightarrow \prod_{v \in V_{\infty}^k} H^i(k_v , M),
$$
where $V_{\infty}^k$ is the set of all real places of $k$, is an isomorphism for any $i \geqslant 3$ (see \cite[Ch. II, 6.3, Theorem B]{Serre-GC}); in particular, the group $H^i(k , M)$ is finite. Since $H^1(C \otimes_k \bar{k} , \mu_n^{\otimes d})$ is finite by \cite[Expos\'e XVI, Th\'eor\`eme 5.2]{SGA4}, it follows that the left-most and the right-most terms in the exact sequence (\ref{E:X0}) are finite. So, the middle term is also finite for all $i \geqslant 4$, as required. \hfill $\Box$

\vskip2mm

\section{Proof of Theorems \ref{T:genus} and \ref{T:HP}.}\label{S:5}

\noindent {\bf 1. The finiteness of genus.} Theorem \ref{T:genus} will be derived from the following result.
\begin{thm}\label{T:genus1}
Let $K$ be a 2-dimensional global field of characteristic $\neq 2$, and let $G = \mathrm{Spin}_n(q)$, where $q$ is a nondegenerate quadratic form over $K$ of dimension $n \geqslant 5$. Then the number of $K$-isomorphism classes of the spinor groups $G' = \mathrm{Spin}_n(q')$ of nondegenerate $n$-dimensional quadratic forms $q'$ over $K$ that have the same isomorphism classes of maximal $K$-tori as $G$, is finite.
\end{thm}

To prove Theorem \ref{T:genus1}, we need the following general statement.
\begin{thm}\label{T:Genus-smooth}
{\rm (\cite[Theorem 5]{CRR4})} Let $G$ be an absolutely almost simple simply connected algebraic group over a field $K$, and let $v$
be a discrete valuation of $K$. Assume that the residue field $K^{(v)}$ is finitely generated and that $G$ has good reduction at $v$. Then any $G' \in \gen_K(G)$ also has good reduction at $v$. Moreover, the reduction ${\uG'}^{(v)}$ lies in the genus $\gen_{K^{(v)}}(\uG^{(v)})$ of the reduction $\uG^{(v)}$.
\end{thm}

Let $V$ be a divisorial set of places of our 2-dimensional global field $K$ of characteristic $\neq 2$. Using the fact that $V$ satisfies condition (A), it is not difficult
to see that the set $v \in V$ where $G = \mathrm{Spin}_n(q)$ does not have good reduction, is finite. Picking a divisorial set of places in the complement of this finite set, we may assume without loss of generality that $G$ has good reduction at all $v \in V$. Then according to
Theorem \ref{T:Genus-smooth},  every $G' \in \gen_K(G)$ has good reduction at all $v \in V$. So, Theorem \ref{T:F1} yields the finiteness of the number of $K$-isomorphism classes of the spinor groups $G' = \mathrm{Spin}_n(q')$ lying in $\gen_K(G)$, which is precisely the assertion of Theorem
\ref{T:genus1}.

\vskip2mm

Now, Theorem \ref{T:genus} for quadratic forms of odd dimension immediately follows from Theorem \ref{T:genus1}. To consider the case of isotropic quadratic forms $q$ of even dimension, we need to combine the latter with part (a) of the following proposition, which implies that every group in $\gen_K(G)$ is of the form $\mathrm{Spin}_n(q')$.
\begin{prop}\label{P:reduction}
Let $K$ be an infinite finitely generated field of characteristic $\neq 2$, let $G = \mathrm{Spin}_n(q)$ where $q$ is a nondegenerate quadratic form over $K$ of even dimension $n \geqslant 10$, and $H = \widetilde{{\rm SU}}_m(D , h)$ (universal cover) where $D$ is a central division algebra over $K$ of degree $d \geqslant 1$ with an orthogonal involution $\tau$ and $h$ is a nondegenerate $m$-dimensional $\tau$-hermitian form.

\vskip2mm

\noindent {\rm (a)} \parbox[t]{15cm}{If $d > 1$ and $q$ is $K$-isotropic, then $H \notin \gen_K(G)$.}

\vskip2mm

\noindent {\rm (b)} \parbox[t]{15cm}{If $d > 2$, then $H \notin \gen_K(G)$ for any $q$.}
\end{prop}
(Note that $d$ is a power of $2$.)

\vskip2mm

We will first establish one general fact that involves the notion of {\it generic tori}, which we will now recall for the reader's convenience. Let $\mathscr{G}$ be a semi-simple algebraic group over a field $\mathscr{K}$. Fix a maximal $\mathscr{K}$-torus $\mathscr{T}$ of $\mathscr{G}$, and let $\Phi(\mathscr{G} , \mathscr{T})$ and $W(\mathscr{G} , \mathscr{T})$ denote the corresponding root system and the Weyl group. The natural action of the absolute Galois group $\Ga(\mathscr{K}^{\mathrm{sep}}/\mathscr{K})$, where $\mathscr{K}^{\mathrm{sep}}$ is a fixed separable closure of $\mathscr{K}$, on the character group $X(\mathscr{T})$ of $\mathscr{T}$ gives rise to a group homomorphism
$$
\theta_{\mathscr{T}} \colon \Ga(\mathscr{K}^{\mathrm{sep}}/\mathscr{K}) \to \mathrm{Aut}(\Phi(\mathscr{G} , \mathscr{T}))
$$
that factors through the Galois group $\Ga(\mathscr{K}_{\mathscr{T}}/\mathscr{K})$ of the minimal splitting field of $\mathscr{T}$ in $\mathscr{K}^{\mathrm{sep}}$ inducing an {\it injective} homomorphism $\bar{\theta}_{\mathscr{T}} \colon \Ga(\mathscr{K}_{\mathscr{T}}/\mathscr{K}) \to \mathrm{Aut}(\Phi(\mathscr{G} , \mathscr{T}))$.
We say that $\mathscr{T}$ is {\it generic} over $\mathscr{K}$ if $\theta_{\mathscr{T}}(\Ga(\mathscr{K}^{\mathrm{sep}}/\mathscr{K})) \supset W(\mathscr{T} , \mathscr{G})$.
\begin{prop}\label{P:reduction1}
Let $\mathscr{D}$ be a central division algebra of degree $d > 1$ over a field $\mathscr{K}$ of characteristic $\neq 2$ with an orthogonal involution $\tau$, and let $h$ be a nondegenerate $m$-dimensional $\tau$-hermitian form. Then $\mathscr{H} = \mathrm{SU}_m(\mathscr{D} , h)$ does not contain a maximal $\mathscr{K}$-torus of the form $\mathscr{T} = \mathscr{T}_1 \mathscr{T}_2$ (almost direct product over $\mathscr{K}$) where $\mathscr{T}_1 = \mathbb{G}_m$ and $\mathscr{T}_2$ is isomorphic to a generic torus in a $\mathscr{K}$-group of type $\textsf{D}_{n/2-1}$ with $n = dm$.
\end{prop}
\begin{proof}
We begin with the following lemma.
\begin{lemma}\label{L:reduction}
Notations as in Proposition \ref{P:reduction1}, let $r$ denote the Witt index of $h$. Then for any nontrivial $\mathscr{K}$-split torus $\mathscr{S}$ of $\mathscr{H}$, the centralizer $C_{\mathscr{H}}(\mathscr{S})$ is an almost direct product $\mathscr{S}_0 \mathscr{H}_1 \cdots \mathscr{H}_s$ over $\mathscr{K}$, where $\mathscr{S}_0$ is the central torus of $C_{\mathscr{H}}(\mathscr{S})$ and each $\mathscr{H}_i$ is a semi-simple $\mathscr{K}$-group  either of inner type $\textsf{A}_{\ell - 1}$ with $\ell \leqslant dr$ or of type $\textsf{D}_{\ell}$ with $\ell \leqslant d(m-2)/2$ (and not a triality form when $\ell = 4$).
\end{lemma}
\begin{proof}
Let $\chi_0 = 1, \chi_1, -\chi_1, \ldots , \chi_t, -\chi_t$ be all the weights for the action of $\mathscr{S}(\mathscr{K})$ on $W = \mathscr{D}^m$, and denote by $W_{\chi}$ the weight subspace corresponding to the character $\chi$ of $\mathscr{S}$. Set
$$
W_0 = W_{\chi_0} \ \ \text{and} \ \ W_i = W_{\chi_i} \oplus W_{-\chi_i} \ \ \text{for} \ \ i = 1, \ldots , t.
$$
Then $W$ has the following orthogonal decomposition
$$
W = W_0 \perp W_1 \perp \cdots \perp W_t.
$$
Let
$$
m_0 = \dim_{\mathscr{D}} W_0 \ \ \text{and} \ \ m_i = \dim_{\mathscr{D}} W_{\chi_i} = \dim_{\mathscr{D}} W_{-\chi_i} \ \ \text{for} \ \ i = 1, \ldots t,
$$
so that $m_i \leqslant r \leqslant m/2$ and $\dim_{\mathscr{D}} W_i = 2m_i$. Then
$$
C_{\mathscr{H}}(\mathscr{S}) = \mathscr{F}_0 \times \mathscr{F}_1 \times \cdots \mathscr{F}_t,
$$
where $\mathscr{F}_0 = \mathrm{SU}_{m_0}(\mathscr{D} , h \vert W_0)$ and $\mathscr{F}_i$ for $i > 0$ is the stabilizer of the subspaces $W_{\chi_i}$ and $W_{-\chi_i}$ in $\mathrm{SU}_{2m_i}(\mathscr{D} , h \vert W_i)$. Since $m_i \geqslant 1$ and $t > 0$ as $\mathscr{S}$ is nontrivial, we see that $m_0 \leqslant m - 2$, and therefore $\mathscr{F}_0$ is of type $\textsf{D}_{\ell}$ with $$\ell = dm_0/2 \leqslant d(m-2)/2$$ (obviously, not a triality form). Furthermore, for $i > 0$, the group $\mathscr{F}_i$ is isomorphic to $\mathrm{GL}(W_{\chi_i})$ (or $\mathrm{GL}(W_{-\chi_i})$). So, its semi-simple part is $\mathrm{SL}(W_{\chi_i})$, which is an inner form of type $\textsf{A}_{\ell - 1}$ with $$\ell = dm_i \leqslant dr.$$
\end{proof}

We will now complete the proof of Proposition \ref{P:reduction1}. Assume the contrary, i.e. $\mathscr{H}$ contains a maximal $\mathscr{K}$-torus $\mathscr{T} = \mathscr{T}_1 \mathscr{T}_2$ as in the proposition. Applying the lemma to $\mathscr{S} = \mathscr{T}_1$, we obtain that
$$
\mathscr{T} = \mathscr{T}_1 \mathscr{T}_2 \subset \mathscr{S}_0 \mathscr{H}_1 \cdots \mathscr{H}_s
$$
where $\mathscr{S}_0$ is the central torus of $C_{\mathscr{H}}(\mathscr{T}_1)$ and the $\mathscr{H}_i$'s are semisimple $\mathscr{K}$-groups of the types specified in the lemma. Clearly, $\mathscr{T} \not\subset \mathscr{S}_0$ as otherwise $\mathscr{H}$ would be quasi-split over $\mathscr{K}$, which it is not because $d > 1$. Since $\mathscr{T}_1 \subset \mathscr{S}_0$, we conclude that $\mathscr{T}_2 \not\subset \mathscr{S}_0$. It follows that there exists $i \in \{1, \ldots , s \}$ such that
$$
\mathscr{T}_2 \not\subset \hat{\mathscr{H}}_i := \mathscr{S}_0 \mathscr{H}_1 \cdots \mathscr{H}_{i-1} \mathscr{H}_{i+1} \cdots \mathscr{H}_s.
$$
This means that for the quotient map $\pi_i \colon C_{\mathscr{H}}(\mathscr{T}_1) \to C_{\mathscr{H}}(\mathscr{T}_1)/\hat{\mathscr{H}}_i =: \bar{\mathscr{H}}_i$, we have $\pi_i(\mathscr{T}_2) \neq \{ e \}$. Being generic over $\mathscr{K}$, the torus $\mathscr{T}_2$ is $\mathscr{K}$-irreducible, i.e. contains no proper $\mathscr{K}$-defined subtori. So, the restriction of $\pi_i$ gives an isogeny of $\mathscr{T}_2$ onto a $\mathscr{K}$-subtorus $\bar{\mathscr{T}}_i$ of $\bar{\mathscr{H}}_i$. Moreover, since $\dim \mathscr{T}_2 = \mathrm{rk}\: \mathscr{H} - 1$ and $\mathrm{rk}\: \bar{\mathscr{H}}_i \leqslant \mathrm{rk}\: \mathscr{H} - 1$, we see that $\bar{\mathscr{T}}_i$ is actually a maximal $\mathscr{K}$-torus of $\bar{\mathscr{H}}_i$. Let $\mathscr{L}$ be the common minimal splitting field of $\mathscr{T}_2$ and $\bar{\mathscr{T}}_i$. Since $\mathscr{T}_2$ is $\mathscr{K}$-generic in a group of type $\textsf{D}_{n/2 - 1}$ where
$n = dm$, we have
\begin{equation}\label{E:Ineq10}
\vert \Ga(\mathscr{L}/\mathscr{K}) \vert \geqslant 2^{n/2 - 2} \cdot (n/2 - 1)!.
\end{equation}
Now, assume that $\mathscr{H}_i$ is an inner form of type $\textsf{A}_{\ell - 1}$ with $\ell \leqslant dr$. Since $r \leqslant m/2$, we obtain
$$
\vert \Ga(\mathscr{L}/\mathscr{K}) \vert \leqslant \vert W(\bar{\mathscr{H}}_i , \bar{\mathscr{T}}_i) \vert = \ell ! \leqslant (n/2)!.
$$
Comparing with (\ref{E:Ineq10}), we obtain the inequality
$$
2^{n/2 - 2} \leqslant n/2,
$$
which fails for all $n \geqslant 10$.

\vskip1mm

Next, let $\mathscr{H}_i$ be of type $\textsf{D}_{\ell}$ with $\ell \leqslant d(m - 2)/2 \leqslant n/2 - 2$ as $d \geqslant 2$. Since it is not a triality form when $\ell = 4$, we have
$$
\vert \Ga(\mathscr{L}/\mathscr{K}) \vert \leqslant 2 \cdot  \vert W(\bar{\mathscr{H}}_i , \bar{\mathscr{T}}_i) \vert = 2^{\ell} \cdot \ell ! \leqslant
2^{n/2 - 2} \cdot (n/2 - 2)!.
$$
So, comparing with (\ref{E:Ineq10}), we obtain the inequality
$$
2^{n/2 - 2} \cdot (n/2 - 1)! \leqslant 2^{n/2 - 2} \cdot (n/2 - 2)!
$$
which never holds for $n \geqslant 6$.

Thus, the assumption that $\mathscr{H}$ contains a maximal $\mathscr{K}$-torus $\mathscr{T} = \mathscr{T}_1 \mathscr{T}_2$ leads to a contradiction in both cases, proving our claim.
\end{proof}

\vskip2mm

\noindent {\it Proof of Proposition \ref{P:reduction}.} (a) Write $q = q_1 \perp q_2$, where $q_1$ is a 2-dimensional hyperbolic form. Then $T_1 = \mathrm{Spin}_2(q_1)$ is a 1-dimensional split torus $\mathbb{G}_m$. Since $K$ is infinite and finitely generated, we can find a maximal $K$-torus $T_2$ of $\mathrm{Spin}_{n-2}(q_2)$ that is generic over $K$ (cf. \cite{Pr-Rap1}, \cite{Pr-Rap2}). Consider the maximal $K$-torus $T = T_1T_2$ of $G$. Then it follows from Proposition \ref{P:reduction1} that $T$ is not isomorphic to a maximal $K$-torus of $H$, hence $H \notin \gen_K(G)$.

\vskip1mm

(b) Again, write $q = q_1 \perp q_2$ where $q_1 = \langle a_1 , a_2 \rangle$ is a binary form. Set $L = K(\sqrt{-a_1a_2})$. If $L = K$, then our claim follows from part (a). So, we may assume that $L$ is a quadratic extension of $K$. Then $T_1 = \mathrm{Spin}_2(q_1)$ is the 1-dimensional norm torus $\mathrm{R}^{(1)}_{L/K}(\mathbb{G}_m)$. Furthermore, let $T_2$ be a maximal $K$-torus of $\mathrm{Spin}_{n-2}(q_2)$ that is generic over $L$ (cf. \cite{Pr-Rap1}, \cite{Pr-Rap2}). Then the maximal $K$-torus $T = T_1T_2$ is not isomorphic to a maximal $K$-torus of $H$. Indeed, assume the contrary, i.e. $T \subset H$. Note that over $L$, the torus $T$ is an almost direct product $T'_1T'_2$ where $T'_1$ is a 1-dimensional split torus and $T'_2$ is isomorphic to a generic torus in a group of type $\textsf{D}_{n/2 - 1}$. At the same time, since $d > 2$, the group $H$ is $L$-isomorphic to $\widetilde{SU}_{m'}(\mathscr{D} , h')$ where $\mathscr{D}$ is a central division algebra over $\mathscr{K} = L$ of degree $d' > 1$ and $h'$ is an $m'$-dimensional hermitian form (note that $d'm' = dm$). But this contradicts Proposition \ref{P:reduction1}. Thus, again $H \notin \gen_K(G)$. \hfill $\Box$

\vskip2mm

\noindent {\bf Remark 5.6.} It follows from Proposition \ref{P:reduction}(b) that in order to complete the proof of Theorem \ref{T:genus} for all even-dimensional form, it would be sufficient to demonstrate that $\gen_K(G)$ cannot contain a group of the form $H = \widetilde{\mathrm{SU}}_m(D , h)$ where $D$ is a central {\it quaternion} division algebra over $K$. The proof of Proposition \ref{P:reduction}(b) given above
shows that this boils down to proving that for any quaternion division algebra $D$ over $K$, the given quadratic form $q$ contains a binary subform $q_1 = \langle a_1 , a_2 \rangle$ such that the field $K(\sqrt{-a_1a_2})$ does not split $D$, which seems very likely at least when the dimension of $q$ is sufficiently large (note that those $d \in K^{\times} \setminus {K^{\times}}^2$ for which $K(\sqrt{d}) \hookrightarrow D$ must be represented by the ternary quadratic form associated with $D$).

\addtocounter{thm}{1}

\vskip2mm

\noindent {\bf 2. Properness of the global-to-local map.} Let $K$ be a 2-dimensional global field of characteristic $\neq 2$, and let $V$ be a divisorial set of places. As we already mentioned, for each $i \geqslant 1$, the kernel
$$
\Omega_i = \ker\left(H^i(K , \mu_2) \longrightarrow \prod_{v \in V} H^i(K_v , \mu_2)\right)
$$
is contained in $H^i(K , \mu_2)_V$, hence is finite\footnote{We can assume without loss of generality that $\mathrm{char}\: K^{(v)} \neq 2$ for all $v \in V$.} (cf. \S 4.1). Applying now Theorem \ref{T:F2-1}, we conclude that for $q$ a nondegenerate quadratic form in $n \geqslant 5$ variables and $G = \mathrm{SO}_n(q)$, the map
$$
\theta_{G , V} \colon H^1(K , G) \longrightarrow \prod_{v \in V} H^1(K_v , G)
$$
is proper, which is precisely the assertion of Theorem \ref{T:HP}.

To put this result in a more general context, we recall the following general observation \cite{CRR2}, \cite{R-ICM}: {\it Let $G$ be an absolutely almost simple simply connected algebraic group over a field $K$, and let $V$ be a set of discrete valuations of $K$ that satisfies condition {\rm (A)}. Assume that for any finite subset $S \subset V$, the set of $K$-isomorphism classes of inner $K$-forms of $G$ having good reduction at all $v \in V \setminus S$ is finite. Then for the corresponding adjoint group $\overline{G}$, the map $$\theta_{\overline{G} , V} \colon H^1(K , \overline{G}) \longrightarrow \prod_{v \in V} H^1(K_v , \overline{G})$$ is proper.} This means, in particular, that for $n$ odd, Theorem \ref{T:HP}
{\it directly} follows from Theorem \ref{T:F1}. On the other hand, for $n$ even, the latter implies neither Theorem \ref{T:HP} nor the corresponding fact for the simply connected and adjoint groups $\widetilde{G} = \mathrm{Spin}_n(q)$ and $\overline{G} = \mathrm{PSO}_n(q)$ (thus, in this case Theorem \ref{T:HP} is an independent result). Later, we will see that the analogues of Theorem \ref{T:HP} are valid for special unitary groups of hermitian forms over a quadratic extension of the base (2-dimensional global) field and for the group of type $\textsf{G}_2$ over such a field (note that all these groups are simply connected) --- see Theorems \ref{T:U3} and \ref{T:G2-1}(iii). We will close this subsection with some additional results on the properness of the global-to-local map for simply connected groups.

\vskip3mm

\begin{thm}\label{T:HP1}
Let $K$ be a 2-dimensional global field,  $V$ be its divisorial set of places, and $m$ be a square-free integer prime to $\mathrm{char}\: K$.
%and such that $K$ contains a primitive $m$-th root of unity.
Then for a central simple $K$-algebra $A$ of degree $m$ and $G = \mathrm{SL}_{1 , A}$, the map $$\theta_{G , V} \colon H^1(K , G) \longrightarrow \prod_{v \in V} H^1(K_v , G)$$ is proper.
\end{thm}
\begin{proof}
We may assume without loss of generality that $\mathrm{char}\: K^{(v)}$ is prime to $m$ for all $v \in V$. It is well-known (cf. \cite[29.4]{KMRT}) that for any extension $F/K$ there is a bijection
$$
\nu_F \colon H^1(F , G) \to F^{\times} / \mathrm{Nrd}_{A_F/F}(A_F^{\times}) \ \ \text{where} \ \ A_F = A \otimes_K F,
$$
and this bijection is functorial in $F$. So, letting $A_v = A_{K_v}$ for $v \in V$, we need to show
that the map
$$
\iota \colon K^{\times}/ \mathrm{Nrd}_{A/K}(A^{\times}) \to \prod_{v \in V} K_v^{\times} / \mathrm{Nrd}_{A_v/K_v}(A_v^{\times})
$$
has finite kernel. For any field extension $F/K$, we have a homomorphism
$$
F^{\times} \to H^3(F , \mu_m^{\otimes 2}), \ \ a \mapsto [A_F] \cup \chi_{m , a},
$$
where $[A_F]$ is the class of $A_F$ in ${}_m\Br(F) = H^2(F , \mu_m)$. Since $m$ is square-free, according to
\cite[Theorem 12.2]{MS}, the kernel of this map  coincides with $\mathrm{Nrd}_{A_F/F}(A_F^{\times})$, so we have an {\it injective} homomorphism
$$
\delta_F \colon F^{\times} / \mathrm{Nrd}_{A_F/F}(A_F^{\times}) \to H^3(F , \mu_m^{\otimes 2}).
$$
We then have the following commutative diagram
%$$
%\begin{tikzcd}
%K^{\times} / \mathrm{Nrd}_{A/K}(A^{\times}) \arrow[r, "\iota"] \arrow[d, "\delta_K"'] & \displaystyle \prod_{v \in V} K_v^{\times} / \mathrm{Nrd}_{A_v/K_v}(A_v^{\times}) \arrow[d, "\Delta_V \ \ \ \ \ \ \ "] \\ H^3(K , \mu_m^{\otimes 2}) \arrow[r, "\eta"] & \displaystyle \prod_{v \in V} H^3(K_v , \mu_m^{\otimes 2})
%\end{tikzcd}
%$$
$$
\xymatrix{K^{\times} / \mathrm{Nrd}_{A/K}(A^{\times}) \ar[r]^{\iota} \ar[d]_{\delta_K} & {\small \prod_{v \in V} K_v^{\times} / \mathrm{Nrd}_{A_v/K_v}(A_v^{\times})} \ar[d]^{{\small \Delta_V}} \\ H^3(K , \mu_m^{\otimes 2}) \ar[r]^{\eta} & \prod_{v \in V} H^3(K_v , \mu_m^{\otimes 2})}
$$
where $\Delta_V = \prod_{v \in V} \delta_{K_v}$. The fact that the unramified cohomology group $H^3(K , \mu_m^{\otimes 2})_V$ is finite (Theorem \ref{T:H3-Finite}) implies that $\ker \eta$ is finite. Then using the injectivity of $\delta_K$, we conclude that $\ker \iota$ is finite as well, as needed.
\end{proof}

\vskip2mm

We now turn to spinor groups. Let $q$ be a nondegenerate  quadratic form of dimension $\geqslant 3$ over a field $K$ of characteristic $\neq 2$, and let $\pi \colon \mathrm{Spin}_n(q) \to \mathrm{SO}_n(q)$ be the canonical $K$-isogeny. For a field extension $F/K$, we let $$\pi_F \colon H^1(F , \mathrm{Spin}_n(q)) \to H^1(F , \mathrm{SO}_n(q))$$ be the map induced by $\pi$. Assume now that $K$ is equipped with a set $V$ of discrete valuations such that the map $\theta \colon H^1(K , \mathrm{SO}_n(q)) \to \prod_{v \in V} H^1(K_v , \mathrm{SO}_n(q))$ is proper. Then, analyzing the commutative diagram
$$
\xymatrix{H^1(K , \mathrm{Spin}_n(q)) \ar[r]^{\tilde{\theta}} \ar[d]_{\pi_K} & \displaystyle \prod_{v \in V} H^1(K_v , \mathrm{Spin}_n(q)) \ar[d]^{\Pi} \\ H^1(K , \mathrm{SO}_n(q)) \ar[r]^{\theta} & \displaystyle \prod_{v \in V} H^1(K_v , \mathrm{SO}_n(q))}
$$
%$$
%\begin{array}{ccc}
%H^1(K , \mathrm{Spin}_n(q)) & \stackrel{\tilde{\theta}}{\longrightarrow} & \displaystyle \prod_{v \in V} H^1(K_v , \mathrm{Spin}_n(q)) \\
%\pi_K \downarrow & &  \downarrow \Pi \\
%H^1(K , \mathrm{SO}_n(q)) & \stackrel{\theta}{\longrightarrow} & \displaystyle \prod_{v \in V} H^1(K_v , \mathrm{SO}_n(q))
%\end{array}
%$$
where $\Pi = \prod_{v \in V} \pi_{K_v}$, and using twisting, we see that in order to prove that $\tilde{\theta}$ is proper, it is enough to prove that for {\it any} nondegenerate $n$-dimensional quadratic form $q'$ over $K$, the set
\begin{equation}\label{E:Set}
\text{\brus Sh}(\mathrm{Spin}_n(q') , V) \cap \ker \pi'_K
\end{equation}
(where $\pi'_K$ is the map analogous to $\pi_K$ for the quadratic form $q'$) is finite. For a field extension $F/K$, combining the standard exact sequence of cohomology corresponding to the exact sequence
$$
1 \to \mu_2 \longrightarrow \mathrm{Spin}_n(q') \stackrel{\pi'}{\longrightarrow} \mathrm{SO}_n(q') \to 1
$$
with the description of the spinor norm (cf. \cite[13.30, 13.31]{KMRT}), one easily obtains a (functorial) identification of $\ker \pi'_F$ with
the quotient $F^{\times}/\mathrm{Sn}(q' , F)$, where $\mathrm{Sn}(q' , F)$ denotes the image of the spinor norm $\mathrm{SO}_n(q')_F$. Furthermore, one observes that this identification induces an injection of the set (\ref{E:Set}) into the kernel of the map
$$
\nu_{q' , V} \colon K^{\times} / \mathrm{Sn}(q' , K) \longrightarrow \prod_{v \in V} K_v^{\times} / \mathrm{Sn}(q' , K_v).
$$
Thus, if $\theta$ is known to be proper, then to establish the properness of $\tilde{\theta}$, it is sufficient to prove that the kernel of $\nu_{q' , V}$ is finite for all nondegenerate $n$-dimensional quadratic forms $q'$ over $K$. (Of course, $\ker \nu_{q' , V}$ is automatically trivial if $q'$ is $K$-isotropic, so the properness of $\theta$ immediately implies that of $\tilde{\theta}$ {\it if} every $n$-dimensional quadratic form over $K$ is isotropic, i.e. $n$ is greater than the $u$-invariant of $K$.)

Here we would like to point our the following partial result on the finiteness of the set (\ref{E:Set}).
\begin{prop}\label{P:HP2}
Let $K$ be a 2-dimensional global field of characteristic $\neq 2$, and $V$ be its divisorial set of places. If $q$ is a quadratic Pfister form over $K$ of dimension $n = 2^d$, then the intersection
$$
\text{\brus Sh}(\mathrm{Spin}_n(q) , V) \cap \ker \pi_K
$$
is finite.
\end{prop}
\begin{proof}
According to the discussion prior to the statement of the proposition, it is enough to show that the map
$$
\eta_{q , V} \colon K^{\times}/\mathrm{Sn}(q , K) \to \prod_{v \in V} K_v^{\times} / \mathrm{Sn}(q , K_v)
$$
has finite kernel. By our assumption, $q$ is a Pfister form $\langle\!\langle a_1, \ldots , a_d \rangle\!\rangle$ for some $a_1, \ldots , a_d \in
K^{\times}$. For any field extension $F/K$, using the fact that the nonzero values of $q$ over $F$ form a subgroup of $F^{\times}$ (see \cite[Ch. X, Theorem 1.8]{Lam}), it is easy to see that $a \in F^{\times}$ belongs to $\mathrm{Sn}(q , F)$ if and only if the Pfister form $q \perp (-a q) = \langle\!\langle a_1, \ldots , a_d , a \rangle\!\rangle$ is isotropic, hence hyperbolic (cf. \cite[Ch. X, Theorem 1.7]{Lam}). The latter is equivalent to the symbol
$$
(a_1, \ldots , a_d, a) \in k_{d+1}(F) = K_{d+1}(F) / 2 \cdot K_{d+1}(F),
$$
or the cup-product
$$
\chi_{a_1} \cup \cdots \cup \chi_{a_d} \cup \chi_a \in H^{d+1}(F , \mu_2),
$$
being trivial. So, the map
$$
F^{\times} \to H^{d+1}(F , \mu_2), \ \ a \mapsto \chi_{a_1} \cup \cdots \cup \chi_{a_d} \cup \chi_{a},
$$
gives rise to an {\it injective} map
$$
\delta_F \colon F^{\times} / \mathrm{Sn}(q , F) \to H^{d+1}(F , \mu_2).
$$
We then have the following commutative diagram
$$
\xymatrix{K^{\times} / \mathrm{Sn}(q , K) \ar[r]^{\eta_{q , V}} \ar[d]_{\delta_K} & \displaystyle \prod_{v \in V} K_v^{\times} / \mathrm{Sn}(q , K_v) \ar[d]^{\Delta_V} \\ H^{d+1}(K , \mu_2) \ar[r]^{\iota_{d+1}} & \displaystyle \prod_{v \in V} H^{d+1}(K_v , \mu_2)}
$$
%$$
%\begin{array}{ccc}
%K^{\times} / \mathrm{Sn}(q , K) & \stackrel{\eta_{q , V}}{\longrightarrow} & \displaystyle \prod_{v \in V} K_v^{\times} / \mathrm{Sn}(q , K_v) \\
%\delta_K \downarrow & & \downarrow \Delta_V \\
%H^{d+1}(K , \mu_2) & \stackrel{\iota_{d+1}}{\longrightarrow} & \prod_{v \in V} H^{d+1}(K_v , \mu_2)
%\end{array}
%$$
where $\Delta_V = \prod_{v \in V} \delta_{K_v}$. As we already mentioned several times, the finiteness of the unramified cohomology group $H^{d+1}(K , \mu_2)_V$ (see \S 4.1), implies the finiteness of $\ker \iota_{d+1}$. So, the injectivity of $\delta_K$ yields the finiteness of  $\ker \eta_{q , V}$, as required.
\end{proof}

\vskip5mm

\section{The finiteness of unramified $H^3$ for 2-dimensional global fields: characteristic zero case}\label{S:H3}

To complete the proof of Theorem \ref{T:F1} and its consequences, we need to prove the following.
\begin{thm}\label{T:H3-Finite}
Let $K$ be a 2-dimensional global field, $n \geqslant 1$ be an integer prime to $\mathrm{char}\: K$, and $V$ a divisorial set of places
such that $\mathrm{char}\: K^{(v)}$ is prime to $n$ for all $v \in V$. Then
the unramified cohomology group $H^3(K , \mu_n^{\otimes 2})_V$ is finite.
\end{thm}

\begin{cor}\label{C:mu2}
Let $K$ be a 2-dimensional global field of characteristic $\neq 2$, and $V$ be a divisorial set of places such that $\mathrm{char}\:
K^{(v)} \neq 2$ for all $v \in V$. Then the unramified cohomology group $H^3(K , \mu_2)_V$ is finite.
\end{cor}

In the current section, we will treat the characteristic zero case, where $K$ is the function field of a smooth geometrically integral curve $C$ over a number field $k$, following the referee's suggestions; the positive characteristic case will be handled at the end of \S \ref{S:Fm}. We begin by considering the case of a projective curve. So, let $\widetilde{C}$ be a smooth {\it projective} geometrically integral curve over a number field $k$ with function field $K = k(\widetilde{C})$, and let $V_0(\widetilde{C})$ be the set of geometric places of $K$
%discrete valuations of $K$
corresponding to the closed points of $\widetilde{C}.$
%(the geometric places of $K$).
It was observed by Colliot-Th\'el\`ene \cite[proof of Theorem 6.2]{CTFinitude} and (independently) by Berhuy \cite[Example 6]{Berhuy} that Kato's cohomological Hasse principle \cite{Kato} implies the following finiteness result.

\begin{thm}\label{T:CT-Ber}
In the above notations, the unramified cohomology group $H^3(K , \mu_n^{\otimes 2})_{V_0(\widetilde{C})}$ is finite, for any $n \geqslant 1$.
\end{thm}
\begin{proof}
Let $V^k$ be the set of (the equivalence classes of) all valuations of $k$. Then according to \cite[Theorem 0.8]{Kato}, the map
$$
H^3(k(\widetilde{C}) , \mu_n^{\otimes 2}) \longrightarrow \prod_{v \in V^k} H^3(k_v(\widetilde{C}_v) , \mu_n^{\otimes 2}),
$$
where $\widetilde{C}_v = \widetilde{C} \times_k k_v$, is injective. Clearly, for $v \in V^k$, the restriction map takes $H^3(k(\widetilde{C}) ,
\mu_n^{\otimes 2})_{V_0(\widetilde{C})}$ into $H^3(k_v(\widetilde{C}) , \mu_n^{\otimes 2})_{V_0(\widetilde{C}_v)}$. But it is shown in \cite[Corollary 2.9]{Kato} that the latter group is finite for all $v$ (see also Theorem \ref{T:IRap} below), and is in fact trivial if $v$ is nonarchimedean and $\widetilde{C}_v$ has good reduction at $v$. Since $\widetilde{C}$ has good reduction at almost all $v \in V^k$, the required finiteness follows.
\end{proof}

Now, given an arbitrary smooth geometrically integral curve $C$ over a number field $k$, we let $\widetilde{C}$ denote the (unique) smooth geometrically integral projective curve over $k$ containing $C$ as an open subset. Then Theorem \ref{T:CT-Ber} immediately gives the finiteness of the unramified cohomology group $H^3(K , \mu_n^{\otimes 2})$ for any divisorial set of places $V$ of $K = k(C) = k(\widetilde{C})$ that contains
$V_0(\widetilde{C})$. In order to establish the finiteness for {\it any} divisorial set, we need to describe the relationship between the unramified
cohomology of a scheme and that of the complement of a closed subscheme of codimension 1. This relationship, which essentially follows from the construction of Kato's complex in \cite{Kato}, was first observed by Colliot-Th\'el\`ene \cite[\S 2]{CTFinitude}.

Let $X$ be an excellent noetherian scheme and $n$ an integer invertible on $X$. For a point $x \in X$, we let $\kappa(x)$ denote its residue field, and let $X_p$ be the set of points of dimension $p$. For any integers $i, j$, Kato \cite{Kato} constructs a homological complex $C_n^{i,j}(X)$
$$
\cdots \to \bigoplus_{x \in X_p} H^{p+ i} (\kappa(x), \mu_n^{\otimes (p + j)}) \to \bigoplus_{x \in X_{p-1}} H^{p+i-1} (\kappa(x), \mu_n^{\otimes (p+j-1)}) \to \cdots \to \bigoplus_{x \in X_0} H^{i} (\kappa(x), \mu_n^{\otimes j}),
$$
where the term $\oplus_{x \in X_p} H^{p+ i} (\kappa(x), \mu_n^{\otimes (p + j)})$ is placed in degree $p$.
The differentials
$$
\partial_p \colon \bigoplus_{x \in X_{p}} H^{p+i}(\kappa(x), \mu_n^{\otimes (p + j)}) \to \bigoplus_{x \in X_{p-1}} H^{p+i-1}(\kappa(x), \mu_n^{\otimes (p + j-1)})
$$
are defined as follows. Let $x \in X_p$ and set $Z_x = \overline{ \{ x \}}$ to be the closure of $x$ in $X$. Then each point $y$ of codimension 1 on $Z_x$ corresponds to a point in $X_{p-1}.$ Let $y_1, \dots, y_s$ be the points on the normalization $\tilde{Z}_x$ lying above $y.$ The local ring at each $y_k$ is a discrete valuation ring, yielding a discrete valuation on the function field $\kappa(x)$ of $\tilde{Z}_x.$ Let
$$
\partial^x_{y_k} \colon H^{p+i} (\kappa(x), \mu_n^{\otimes (p+j)}) \to H^{p+i-1} (k(y_k), \mu_n^{\otimes (p+j-1)})
$$
be the corresponding residue map. One then defines
$$
\partial^x_y = \sum_{k=1}^s {\rm Cor}_{\kappa(y_k)/ \kappa (y)} \circ \partial^x_{y_k},
$$
where ${\rm Cor}$ is the corestriction map. The differential $\partial_p$ is the direct sum of all such maps. In \cite[Proposition 1.7]{Kato}, Kato verified that this construction produces a complex. Notice that if $X$ is a noetherian integral {\it normal} scheme of dimension $d$ with function field $K$, then the construction of the differentials implies that the degree $d$ homology $H_d (C^{i,j}_n(X))$ of $C^{i,j}_n(X)$ is precisely the unramified cohomology $H^{d+j}(K, \mu_n^{\otimes (d+j)})_{V(X)}$, where $V(X)$ is the set of discrete valuations of $K$ corresponding to the prime divisors of $X$.
%For clarity of notation in the discussion below, let us denote this group by $H^{d+j}_{ur}(K/X, \mu_m^{\otimes (d+j)}).$

Suppose now that $X$ is an excellent integral noetherian normal scheme with function field $K$ and $Y \subset X$ is a {\it normal} closed subscheme of pure codimension 1. So, if $Y = \bigcup_{j \in J} Y_j$ is the decomposition into irreducible components, then each $Y_j$ has codimension 1 in $X$, and we let $\kappa_j$ denote the function field of $Y_j$. Set $U = X \setminus Y$, and let $n$ be a positive integer invertible on $X$. Then for any $i \geqslant 1$ and $\ell$ we  have a canonical embedding
$$
s \colon H^i(K , \mu_n^{\otimes \ell})_{V(X)} \longrightarrow H^i(K , \mu_n^{\otimes \ell})_{V(U)}.
$$
On the other hand, for each $j \in J$, one has the residue map $H^i(K , \mu_n^{\otimes \ell}) \to H^{i-1}(\kappa_j , \mu_n^{\otimes (\ell-1)})$, and we let
$$
r \colon H^i(K , \mu_n^{\otimes \ell}) \longrightarrow \bigoplus_{j \in J} H^{i-1}(\kappa_j , \mu_n^{\otimes (\ell-1)})
$$
denote the direct sum of the residue maps.
\begin{prop}\label{P:Open}
The above maps $r$ and $s$ give the following exact sequence of unramified cohomology groups
\begin{equation}\label{E:Open}
0 \to H^i(K , \mu_n^{\otimes \ell})_{V(X)} \stackrel{s}{\longrightarrow} H^i(K , \mu_n^{\otimes \ell})_{V(U)} \stackrel{r}{\longrightarrow} \bigoplus_{j \in J} H^{i-1}(\kappa_j , \mu_n^{\otimes (\ell-1)})_{V(Y_j)}.
\end{equation}
\end{prop}
\begin{proof}
Let $d = \dim X.$ The initial segment of the Kato complex for $X$ is
$$
\partial_X \colon H^i (K, \mu_n^{\otimes \ell}) \to \left( \bigoplus_{x \in U_{d-1}} H^{i-1}(\kappa(x), \mu_n^{\otimes (\ell -1)}) \right) \oplus \left( \bigoplus_{j \in J} H^{i-1} (\kappa_j, \mu_n^{\otimes (\ell-1)}) \right)
$$
and for $U$ it is
$$
\partial_U \colon H^i (K, \mu_n^{\otimes \ell}) \to  \bigoplus_{x \in U_{d-1}} H^{i-1}(\kappa(x), \mu_n^{\otimes (\ell -1)}).
$$
As noted above, the normality of $X$ implies that $$H^i(K , \mu_n^{\otimes \ell})_{V(X)} = \ker \partial_X \ \ \text{and} \ \  H^i(K , \mu_n^{\otimes \ell})_{V(U)} = \ker \partial_U.$$ At the next stage of the Kato complex for $X$, there are residue maps
$$
\partial_{Y_j} \colon H^{i-1}(\kappa_j, \mu_n^{\otimes (\ell -1)}) \to \bigoplus_{x \in (Y_j)_{c-1}} H^{i-2}(\kappa(x), \mu_n^{\otimes (\ell-2)})
$$
for all $j \in J$, where $c = \dim Y_j = d-1$, and the assumption that $Y$ is normal gives $$H^{i-1}(\kappa_j , \mu_n^{\otimes (\ell-1)})_{V(Y_j)} = \ker \partial_{Y_j}.$$ Now, the fact that we have a complex implies that the map $r$ takes $H^i(K , \mu_n^{\otimes \ell})_{V(U)}$ into
$$\displaystyle \bigoplus_{j \in J} H^{i-1}(\kappa_j , \mu_n^{\otimes (\ell-1)})_{V(Y_j)},$$ and the exactness of (\ref{E:Open}) then follows easily from the construction.
\end{proof}

\vskip1mm

\noindent {\it Proof of Theorem \ref{T:H3-Finite}.} First, let $C$ be any smooth geometrically integral curve over $k$ such that $K = k(C)$, and let $\widetilde{C}$ be a smooth projective geometrically integral curve over $k$ containing $C$. Suppose that we are given smooth models $\mathcal{C} \subset \widetilde{\mathcal{C}}$ of $C$ and $\tilde{C}$, respectively, over a suitable open subsect $T \subset {\rm Spec}(\Z)$ such that $\tilde{\mathcal{C}} \setminus \mathcal{C}$ is of pure codimension 1. The inclusion $$H^3 (K, \mu_n^{\otimes 2})_{V(\widetilde{\mathcal{C}})} \subset H^3(K, \mu_n^{\otimes 2})_{V_0 (\widetilde{C})}$$ in conjunction with Theorem \ref{T:CT-Ber} implies that $H^3 (K, \mu_n^{\otimes 2})_{V(\widetilde{\mathcal{C}})}$ is finite.

Now, let $$\widetilde{\mathcal{C}} \setminus \mathcal{C} = \bigcup_{j \in J} Y_j$$ be the decomposition into irreducible components, and let
$\kappa_j$ be the function field of $Y_j.$ By construction, each $\kappa_j$ is a number field and $Y_j$ is the spectrum of the ring of $S_j$-integers in $\kappa_j$ for some finite subset $S_j \subset V^{\kappa_j}$ containing $V_{\infty}^{\kappa_j}$. Set $T_j = V^{\kappa_j} \setminus S_j$. Then
$H^2(\kappa_j, \mu_n)_{V(Y_j)}$ coincides with the unramified Brauer group ${}_n\Br(\kappa_j)$, the finiteness of which easily follows from the Artin-Hasse-Brauer-Noether Theorem (cf. \cite[3.6]{CRR1}, \cite[3.5]{CRR3}). Let us consider now the exact sequence of Proposition \ref{P:Open} for $X = \widetilde{\mathcal{C}}$ and $U = \mathcal{C}$ for $i = 3$ and $\ell = 2$. Since the finiteness of
$$
H^3 (K, \mu_n^{\otimes 2})_{V(\widetilde{\mathcal{C}})}  \ \ \text{and} \ \ \bigoplus_{j \in J} H^{2}(\kappa_j , \mu_n)_{V(Y_j)}
$$
has already been established, we obtain the finiteness of $H^3 (K, \mu_n^{\otimes 2})_{V(\mathcal{C})}$. This proves the Theorem for the divisorial
set $V = V(\mathcal{C})$. However, it is easy to see that any divisorial set of valuations $V$ of $K$ contains a set of the form $V(\mathcal{C})$ for a suitable curve $C$ over $k$ and its model $\mathcal{C}$ as above, completing the argument in the general case. \hfill $\Box$

\vskip2mm

As we already mentioned, in the Appendix (\S\ref{S:Append}) we will give an alternative (in fact, our initial) proof of Theorem \ref{T:H3-Finite} in characteristic zero, which does not use Kato's cohomological Hasse principle, but rather is inspired by Jannsen's \cite{Jann} approach to the latter.
%elaborates on Jannsen's approach \cite{Jann} to the proof of the latter.

\vskip5mm

\section{Proofs of Theorem \ref{T:F3} and of Theorem \ref{T:H3-Finite} in positive characteristic characteristic}\label{S:Fm}

\vskip5mm

The reason we treat these two results in the same section is that both arguments ultimately depend on the same techniques, namely finiteness theorems in \'etale cohomology and Bloch-Ogus theory. We begin with the proof of Theorem \ref{T:F3}, which relies on the following.
\begin{thm}\label{T:IRap}
{\rm (\cite[Theorem 1.3(a)]{IRap})}
Let $C$ be a smooth geometrically integral curve  over a field $k$, $V_0$ the set of places of $K = k(C)$ associated with the closed points of $C$, and $m \geqslant 1$ an integer prime to $\mathrm{char}\: k$. If $k$ satisfies condition $(\mathrm{F}'_m)$ (see \S1), then for any $i \geqslant 1$ and any $j$, the unramified cohomology group $H^i(K , \mu_n^{\otimes j})_{V_0}$ is finite.
\end{thm}

We refer the reader to \cite{IRap} for a discussion of properties and examples of fields satisfying condition $(\mathrm{F}'_m)$. A sketch of the proof of Theorem \ref{T:IRap} will be given below after briefly recalling some facts from Bloch-Ogus theory that are also needed for the proof of
%a brief account of the facts from the Bloch-Ogus theory that are needed to prove this theorem as well as
Theorem \ref{T:H3-Finite} in positive characteristic. Here we only mention that the proof of Theorem \ref{T:IRap} makes use of the following statement, which is of independent interest and which we will also need in the proof of Theorem \ref{T:F3}.
\begin{prop}\label{P:IRap}
{\rm (\cite[Theorem 1.1]{IRap})} Let $K$ be a field and $m \geqslant 1$ be an integer prime to $\mathrm{char}\: K$. Assume that $K$ satisfies $(\mathrm{F}'_m)$. Then for any finite Galois module $A$ over $K$ such that $mA = 0$, the groups $H^i(K , A)$ are finite for all $i \geqslant 0$.
\end{prop}

\vskip2mm

\noindent {\it Proof of Theorem \ref{T:F3}.}  We need to check the conditions of Theorem \ref{T:F2} for $V = V_0$. Condition (2) follows immediately from Theorem \ref{T:IRap}. To verify (1), we observe that the argument given in the proof of Proposition \ref{P:Pic} shows that there exists an embedding
$$
\Pic(V_0) / 2 \cdot \Pic(V_0) \hookrightarrow H^2(C , \mu_2).
$$
Then, since the group $H^2(C \otimes_k \bar{k} , \mu_2)$ is finite and of exponent 2 (see \cite[Expos\'e XVI, Th\'eor\`eme 5.2]{SGA4}), the finiteness of $H^2(C, \mu_2)$ follows easily from Proposition \ref{P:IRap} (with condition $(\mathrm{F}'_2)$) and the Hochschild-Serre spectral sequence
$$
E_2^{p,q} = H^p(k, H^q (C \otimes_k \bar{k} , \mu_2)) \Rightarrow H^{p+q}(C, \mu_2).
$$
Alternatively, note that we may assume that $C$ is affine, in which case the fundamental sequence  (\ref{E:X0}) yields the following exact sequence
$$
H^2(k , \mu_2) \longrightarrow H^2(C , \mu_2) \longrightarrow H^1(k , H^1(C \otimes_k \bar{k} , \mu_2)).
$$
The extreme terms are finite in view of condition $(\mathrm{F}'_2)$ and Proposition \ref{P:IRap}, so
%Since the group $H^1(C \otimes_k \bar{k} , \mu_2)$ is finite and of exponent 2, according to Proposition \ref{P:IRap} condition $(\mathrm{F}'_2)$ implies the finiteness of both extreme terms in this sequence. So,
the middle term is also finite, as required. \hfill $\Box$

\vskip2mm

We expect a suitable analogue of Theorem \ref{T:F3} to be valid for all types, and would like to propose the following conjecture.

\vskip2mm

\noindent {\bf Conjecture 7.3.} {\it Let $K = k(C)$ be the function field of a smooth affine geometrically integral curve over a field $k$, and let $V_0$ be the set of discrete valuations associated with the closed points of $C$. Furthermore, let $G$ be an absolutely almost simple simply connected algebraic $K$-group, and let $m$ be the order of the automorphism group of its root system. Assume that $\mathrm{char}\: k$ is prime to $m$ and that $k$ satisfies $(\mathrm{F}'_m)$. Then  the set of $K$-isomorphism classes of $K$-forms of $G$ that have good reduction at all $v \in V_0$ is finite.}

\addtocounter{thm}{1}

\vskip2mm

It is likely that the conclusion should be true under more lax assumptions, e.g. it is probably enough to require condition $(\mathrm{F}'_p)$ for all primes $p \ \vert \ m$ or even a certain subset of the set of such primes.

\vskip2mm

Applying Theorem \ref{T:F2-1} in conjunction with Theorem \ref{T:IRap}, we obtain the following.
\begin{thm}\label{T:HP5}
Let $C$ be a smooth geometrically integral curve over a field $k$ of characteristic $\neq 2$ that satisfies condition $(\mathrm{F}'_2)$, and let $K =
k(C)$ be its function field. Denote by $V_0$ the set of discrete valuations of $K$ associated with the closed points of $C$. Then for a nondegenerate quadratic form $q$ of dimension $n \geqslant 5$ over $K$ and $G = \mathrm{SO}_n(q)$, the map
$$
\theta_{G , V_0} \colon H^1(K , G) \longrightarrow \prod_{v \in V_0} H^1(K_v , G)
$$
is proper.
\end{thm}

We expect that the conclusion of the theorem should be true for any absolutely almost simple group under the assumptions made in Conjecture 6.3.
Here is one result for spinor groups in a more specialized situation. Keep the above notations but now assume that $k$ is a finite extension of the $p$-adic field $\Q_p$.
%Let $k$ be a finite extension of the $p$-adic field $\Q_p$, let $K = k(C)$ be the function field of a smooth geometrically integral curve $C$ over %$k$, and $V_0$ be the set of discrete valuations of $K$ corresponding to the closed points of $C$.
Recall that $k$ satisfies Serre's condition $(\mathrm{F})$ (cf. \cite[Ch. III, \S 4]{Serre-GC}), hence condition $(\mathrm{F}'_m)$
for all $m$. Furthermore, Parimala and Suresh \cite{Par-Sur2} showed that the $u$-invariant of $K$ is $8$, i.e. any quadratic form over $K$ of dimension $\geqslant 9$ is isotropic. So, combining Theorem \ref{T:HP5} with the discussion preceding Proposition \ref{P:HP2}, we obtain the following.
\begin{cor}
In the above notations, for any quadratic form $q$ over $K$ in $n \geqslant 9$ variables and $\tilde{G} = \mathrm{Spin}_n(q)$, the map
$\theta_{\tilde{G} , V_0}$ is proper.
\end{cor}

\vskip2mm

We will now recall several results from Bloch-Ogus theory. For any smooth algebraic variety $X$ over an arbitrary field $F$, Bloch and Ogus \cite{Bl-Og} established the existence of the following cohomological first quadrant spectral sequence
\begin{equation}\label{E-ConiveauSS}
E_1^{p,q}(X/F, \mu_m^{\otimes b}) = \bigoplus_{x \in X^{(p)}} H^{q-p}(\kappa(x), \mu_m^{\otimes (b - p)}) \Rightarrow H^{p+q}(X, \mu_m^{\otimes b}),
\end{equation}
where $X^{(p)}$ denotes the set of points of $X$ of codimension $p$ and the summands are the Galois cohomology groups of the residue fields $\kappa(x)$ (the original statement of Bloch-Ogus was actually given in terms of \'etale homology, with the above version obtained via absolute purity; we refer the reader to \cite{CT-H-K} for a derivation of this spectral sequence that avoids the use of \'etale homology, as well as to \cite{CT-SB} for an extensive discussion of applications of this spectral sequence to the Gersten conjecture).
The spectral sequence yields a complex
\begin{equation}\label{E-BOComplex}
E_1^{\bullet, q}(X/F, \mu_m^{\otimes b}),
\end{equation}
and it is well-known (see, e.g., \cite[Remark 2.5.5]{JannSS}) that the differentials in (\ref{E-BOComplex}) coincide up to sign with the differentials in Kato's complex that was recalled in \S5. The fundamental result of Bloch and Ogus was the calculation of the $E_2$-term of (\ref{E-ConiveauSS}): they showed that
$$
E_2^{p,q}(X/F, \mu_m^{\otimes b}) = H^p (X, \mathcal{H}^q (\mu_m^{\otimes b})),
$$
where $\mathcal{H}^q (\mu_m^{\otimes j})$ denotes the Zariski sheaf on $X$ associated to the presheaf that assigns to an open $U \subset X$ the cohomology group $H^i(U, \mu_m^{\otimes j}).$
The resulting (first quadrant) spectral sequence
\begin{equation}%\label{E-BO}
E_2^{p,q}(X/F, \mu_m^{\otimes b}) = H^p (X, \mathcal{H}^q (\mu_m^{\otimes b})) \Rightarrow H^{p+q}(X, \mu_m^{\otimes b})
\end{equation}
is usually referred to as the {\it Bloch-Ogus spectral sequence} and has the following key properties:

\vskip1mm

\noindent {\rm (a)} $E_2^{p,q} = 0$ for $p > \dim X$ and all $q$;

\vskip1mm

\noindent {\rm (b)} $E_2^{p,q} = 0$ for $p > q$; and

\vskip1mm

\noindent {\rm (c)} \parbox[t]{16cm}{$E_2^{0, q} = H^0 (X, \mathcal{H}^q (\mu_m^{\otimes b}))$ coincides with the unramified cohomology $H^q (F(X), \mu_m^{\otimes b})_{V_0}$ with respect to the geometric places (associated with the prime divisors of $X$) of the function field $F(X)$.}

\vskip3mm

For our applications, we will now state the following proposition, whose proof combines Proposition \ref{P:IRap} with the Hochschild-Serre spectral
sequence.
\begin{prop}\label{P:IRap2}
{\rm (\cite[Corollary 3.2]{IRap})}
Suppose $K$ is a field and $m \geqslant 1$ is integer prime to ${\rm char}~K.$ If $K$ is of type $\Fpm$, then for any smooth geometrically integral algebraic variety $X$ over $K$, the \'etale cohomology groups $H^i(X, \mu_m^{\otimes j})$ are finite for all $i \geqslant 0$ and all $j$.
\end{prop}

\vskip2mm

\noindent {\it Sketch of proof of Theorem \ref{T:IRap}.} Since $\dim C = 1$, property (a) of the Bloch-Ogus spectral sequence gives surjective edge maps $E^i \to E_2^{0,i}$ for all $i \geqslant 1$. So, in view of property (c), we get surjections
$$
H^i (C, \mu_2) \twoheadrightarrow H^i(k(C), \mu_2)_{V_0}
$$
for all $i \geq 1.$ The finiteness of the unramified cohomology groups then follows from Proposition \ref{P:IRap2}.

\vskip2mm

\noindent {\it Proof of Theorem \ref{T:H3-Finite} in positive characteristic.} Here $K$ is the function field $k(X)$ of a smooth geometrically integral surface $X$ over a finite field $k$. Property (b) of the Bloch-Ogus spectral sequence yields an exact sequence
$$
E^3 \to E_2^{0,3} \to E_2^{2,2}.
$$
Hence, using the well-known isomorphism $$H^2(X, \mathcal{H}^2 (\mu_m^{\otimes 2}))\simeq CH^2 (X)/m \cdot CH^2(X),$$ where $CH^2(X)$
is the Chow group of codimension 2 cycles on $X$ (see, e.g., the proof of \cite[Theorem 7.7]{Bl-Og}), together with property (c) above, we obtain the exact sequence
$$
H^3 (X, \mu_m^{\otimes 2}) \to H^3(K, \mu_m^{\otimes 2})_{V_0} \to CH^2 (X)/m \cdot CH^2(X).
$$
The left-most term is finite by Proposition \ref{P:IRap2}. On the other hand, since $X$ is a smooth surface over a finite field, the group $CH^2(X)$ is finitely generated (see \cite[Proposition 4 and Corollaire 7]{CTSS}). It follows that $H^3(K, \mu_m^{\otimes 2})_{V_0}$ is finite, as needed. $\Box$

%\noindent $\bullet$ As the referee points out, the finiteness of $\Pic(V)/2 \cdot \Pic(V)$ can be deduced from well-known finiteness statements in \'etale cohomology.

%Namely, suppose that $U$ is a smooth integral scheme of finite type over $T = {\rm Spec}(A)$, where $A$ is a ring of $S$-integers in a number field $k$ and $S$ is a finite set of places of $k$ containing all the archimedean ones. Assume that $2$ is a unit in $A$.  Let $V$ be the set of discrete valuations of the function field $K$ of $U$ associated with the prime divisors of $U$. Then $\Pic(V)$ is naturally identified with $\Pic(U)$, and the Kummer sequence gives an injection $\Pic (V)/ 2 \cdot \Pic(V) \hookrightarrow \he^1(U, \mu_2).$ On the other hand, it follows from \cite[Ch. II, Theorem 2.13]{MilneADT} and \cite[Th\'eor\`eme 1.1 in ``Th\'eor\`eme de finitudes"]{Del} that all of the \'etale cohomology groups $\he^i(U, \mu_2)$ ($i \geq 0$) are finite.

%\vskip1cm

%\noindent $\bullet$

%To be added to $\S 6$, probably as part (b) of Theorem 6.1: Suppose $X$ is a smooth projective geometrically integral surface over a finite field $\mathbb{F}_q$ of characteristic $\neq 2$ and let $V_0$ be the set of places of $F = \mathbb{F}_q(X)$ associated with the prime divisors of $X$. Then $\Pic(V_0)/2 \cdot \Pic(V_0)$ and $H^i(F, \mu_2)_{V_0}$ $(i \geqslant 0)$ are finite.

%\begin{proof}

%\end{proof}

\vskip5mm

\section{Special unitary groups of hermitian forms over quadratic extensions}\label{S:unitary}

Let $K$ be a field of characteristic $\neq 2$, and let $L/K$ be a quadratic field extension with nontrivial automorphism $\tau$. A choice of a basis of $L$ over $K$ enables us to identify $L$ with $K^2$. Then for any $n \geqslant 1$ the space $L^n$ gets identified with $K^{2n}$, and under this identification, to any
$\tau$-hermitian form $h$ on $L^n$, there naturally corresponds a quadratic form $q = q_h$ on $K^{2n}$ given by $q_h(x) = h(x)$ for $x \in K^{2n} = L^n$. Recall that if $s(x , y)$ is the sesqui-linear form on $L^n \times L^n$ that gives $h$ (i.e. $h(x) = s(x , x)$) then the bilinear form $b$ on $K^{2n} \times K^{2n}$ associated with $q_h$ is given by
$$
b(x , y) = \frac{1}{2} \mathrm{Tr}_{L/K}(s(x , y)) \ \ \text{for} \ \ x , y \in K^{2n} = L^n.
$$
The following result, attributed in \cite[\S 2.9]{LJ-Ger} to Jacobson, is well-known and easy to prove: {\it two $n$-dimensional hermitian forms $h_1$ and $h_2$ are equivalent if and only if the corresponding $2n$-dimensional quadratic forms $q_{h_1}$ and $q_{h_2}$ are equivalent.} In terms of Galois cohomology, this means that the natural map
$$
H^1(K , \mathrm{U}_n(L/K , h)) \longrightarrow H^1(K , \mathrm{O}_{2n}(q_h))
$$
is injective. On the other hand, it is well-known that the map $$H^1(K , \mathrm{SU}_n(L/K , h)) \longrightarrow H^1(K , \mathrm{U}_n(L/K , h))$$ is also injective (the proof repeats verbatim the argument for the injectivity of the map $$H^1(K , \mathrm{SO}_m(q)) \longrightarrow H^1(K , \mathrm{O}_m(q)),$$  cf. \cite[proof of Corollary IV.11.3]{Ber} or \cite[29.E]{KMRT}). It follows that the map
$$
H^1(K , \mathrm{SU}_n(L/K , h)) \longrightarrow H^1(K , \mathrm{SO}_{2n}(q_h))
$$
is injective (in fact, it remains injective also over any extension of $K$). In this section, we will use these facts to prove the analogues of Theorems 1.1-1.4 for the special unitary groups $\mathrm{SU}_n(L/K , h)$ of nondegenerate $n$-dimensional hermitian forms $(n \geqslant 2)$.

\vskip2mm

First, let $G = \mathrm{SU}_n(L/K , h)$ in the above notations, and let $v$ be a discrete valuation of $K$. There are two cases.

\vskip1mm

\noindent {\it Case 1.} \underline{$v$ splits in $L$} (i.e. $L \otimes_K K_v \simeq K_v \oplus K_v$). In this case $G$ is $K_v$-isomorphic to $\mathrm{SL}_n$, hence has good reduction at $v$. At the same time, the corresponding quadratic form $q_{h}$ becomes hyperbolic over $K_v$ and therefore (trivially) $[q_h] \in W_0(K_v)$ in the notations of \S \ref{S:F2}.

\vskip1mm

\noindent {\it Case 2.} \underline{$v$ does not split in $L$} (i.e. $L_v := L \otimes_K K_v$ is a quadratic field extension of $K_v$). Then $G$ has good reduction at $v$ if and only if $L_v/K_v$ is unramified at $v$ and there exists $\lambda \in K_v^{\times}$ such that the hermitian $L_v/K_v$-form $\lambda h$ is equivalent to a hermitian form given by
$$
h'(x_1, \ldots , x_n) = a_1 N_{L_v/K_v}(x_1) + \cdots + a_v N_{L_v/K_v}(x_n)
$$
with $a_i \in U(K_v)$. Then again $[\lambda q_h] \in W_0(K_v)$.

\vskip1mm

\noindent Thus, in all cases, the fact that $G$ has a good reduction at $v$ implies that there exists $\lambda \in K_v^{\times}$ such that $[\lambda q_h] \in W_0(K_v)$, or equivalently the group $H = \mathrm{Spin}_{2n}(q_h)$ has a good reduction at $v$.
It is now easy to derive a unitary analogue of Theorem \ref{T:F1}.
\begin{thm}\label{T:U1}
Let $K$ be a 2-dimensional global field of characteristic $\neq 2$, and let $V$ be a divisorial set of places. Fix a quadratic extension $L/K$, and let $n \geqslant 2$. Then the number of $K$-isomorphism classes of special unitary groups $G = \mathrm{SU}_n(L/K , h)$ of nondegenerate hermitian $L/K$-forms in $n$ variables that have good reduction at all $v \in V$ is finite.
\end{thm}

Indeed, let $G_i = \mathrm{SU}_n(L/K , h_i)$ $(i \in I)$ be an infinite family of pairwise nonisomorphic special unitary groups associated with the quadratic extension $L/K$ such that each $G_i$ has a good reduction at every $v \in V$. Then $H_i = \mathrm{Spin}_{2n}(q_{h_i})$ $(i \in I)$ is a family of spinor groups
each having good reduction at all $v \in V$. Applying Theorem \ref{T:F1}, we conclude that the groups $H_i$ and $H_j$ are $K$-isomorphic for some $i , j \in I$, $i \neq j$. Then for some $\lambda \in K^{\times}$, the quadratic forms $q_{h_i}$ and $\lambda q_{h_j}$ are equivalent. It now follows from Jacobson's
theorem that the hermitian forms $h_i$ and $\lambda h_j$ are equivalent, hence the groups $G_i$ and $G_j$ are isomorphic, a contradiction.

\vskip2mm

\noindent {\bf Remark 8.2.} It follows, for example, from Proposition 5.1 in \cite{CRR3} that $K$ has only finitely many quadratic extensions $L/K$ that are unramified at all $v \in V$ (note that this conclusion remains valid for any finitely generated field $K$ and a divisorial set of places $V$). So, in effect, Theorem \ref{T:U1} yields the finiteness of the set of $K$-isomorphic classes of the special unitary groups  with good reduction at all $v \in V$ of $n$-dimensional nondegenerate hermitian forms associated with {\it all} quadratic extensions $L/K$.

\vskip2mm

\addtocounter{thm}{1}

Turning now to the genus of $G = \mathrm{SU}_n(L/K , h)$, we observe that any $G' \in \gen_K(G)$ is of the form $G' = \mathrm{SU}_n(L/K , h')$. This is clear for $n = 2$, so we assume that $n \geqslant 3$. The group $G$ possesses a maximal $K$-torus $T$ of the form $\mathrm{R}_{L/K}(\mathbb{G}_m)^{n-1}$, so the group $G'$ also has such a maximal $K$-torus. Note that the nontrivial automorphism $\tau \in \Ga(L/K)$ acts on the character group $X(T)$ as multiplication by $(-1)$, and since $- \mathrm{id}$ is not in the Weyl group of the root system of type $A_{n-1}$ $(n \geqslant 3)$, we see that $G'$ is an outer form of the split group of this type and $L$ is the minimal extension of $K$ over which it becomes an inner form. Furthermore, since $G'$ splits over $L$, it cannot involve any noncommutative division $L$-algebra in its description, and therefore it must be of the form $\mathrm{SU}_n(L/K , h')$ (cf. \cite[2.3]{Pl-R}). Now, arguing as in the proof of Theorem \ref{T:genus} (cf. \S \ref{S:ProofTF1}.4) on the basis of Theorem \ref{T:U1}, we obtain the following statement, which is even more complete (in the sense that it has no exceptions) than Theorem \ref{T:genus}.
\begin{thm}\label{T:U2}
Let $K$ be a 2-dimensional global field of characteristic $\neq 2$, and let $G = \mathrm{SU}_n(L/K , h)$, where $L/K$ is a quadratic extension and $h$ is a nondegenerate hermitian form of dimension $n \geqslant 2$ associated with $L/K$. Then the genus $\gen_K(G)$ is finite.
\end{thm}

\vskip2mm

Next, we have the following cohomological statement, which is analogous to Theorem \ref{T:HP}.
\begin{thm}\label{T:U3}
Notations as in Theorem \ref{T:U1}, for $G = \mathrm{SU}_n(L/K , h)$ the map $$H^1(K , G) \to \prod_{v \in V} H^1(K_v , G)$$ is proper.
\end{thm}
Indeed, let $H = \mathrm{SO}_{2n}(q_h)$. Then we have the following commutative diagram
$$
\xymatrix{H^1(K,G) \ar[rr]^{\alpha} \ar[d]_{\beta} & & \prod_{v \in V} H^1(K_v, G) \ar[d]^{\gamma} \\ H^1(K,H) \ar[rr]^{\delta} & & \prod_{v \in V} H^1(K_v, H)}
$$
%$$
%\begin{array}{ccc}
%H^1(K , G) & \stackrel{\alpha}{\longrightarrow} & \prod_{v \in V} H^1(K_v , G) \\
%\beta \downarrow & & \downarrow \gamma \\
%H^1(K , H) & \stackrel{\delta}{\longrightarrow} & \prod_{v \in V} H^1(K_v , H)
%\end{array} .
%$$
According to Theorem \ref{T:HP}, the map $\delta$ is proper. On the other hand, as we pointed out earlier in this section, the map $\beta$ is injective, and the properness of $\alpha$ follows.

\vskip2mm

Finally, we have the following result for special unitary groups over function fields of curves over fields satisfying ${\rm (F'_2)}$, which is analogous to Theorem \ref{T:F3} and which is actually derived from it just like Theorem \ref{T:U1} was derived from Theorem \ref{T:F1}.
\begin{thm}\label{T:U4}
Let $C$ be a smooth geometrically integral curve over a field $k$ of characteristic $\neq 2$ that satisfies condition $(\mathrm{F}'_2)$, and let $K = k(C)$. Denote by $V$ the set of discrete valuations  of $K$ corresponding to the closed points of $C$. Let $L/K$ be a quadratic extension. Then the number of $K$-isomorphism classes of special unitary groups $G = \mathrm{SU}_n(L/K , h)$ of nondegenerate hermitian $L/K$-forms $h$ in $n \geqslant 2$ variables that have good reduction at all $v \in V$ is finite.
\end{thm}

\vskip2mm

\noindent {\bf Remark 8.6.} A similar approach can be used to obtain the analogues of Theorems \ref{T:U1} and \ref{T:U2}-\ref{T:U4} for the (special)
unitary groups $G = \mathrm{SU}_n(D , h)$ of nondegenerate  hermitian forms $h$ of dimension $n \geqslant 2$ over a quaternion division algebra $D$ with center $K$ with the canonical involution. (We note that these are precisely the absolutely almost simple simply connected groups of type $\textsf{C}_n$ that split over a quadratic extension of the base field.) More precisely, let us first fix a central quaternion division algebra $D$ over $K$. Then fixing a basis of $D$ over $K$ enables us to identify $D^n$ with $K^{4n}$, and in terms of this identification, the hermitian form $h$ corresponds to a quadratic form $q_h$ on $K^{4n}$. Furthermore, one has an analogue of Jacobson's theorem: {\it two nondegenerate $n$-dimensional hermitian forms $h_1$ and $h_2$ over $D$ if and only if the corresponding $4n$-dimensional quadratic forms $q_{h_1}$ and $q_{h_2}$ are equivalent.} Next, if $G = \mathrm{SU}_n(D , h)$ has good reduction at a discrete valuation $v$ of $K$ then $D$ is unramified at $v$ and $H = \mathrm{Spin}_{4n}(b_h)$ has good reduction at $v$. Using these remarks, one can repeat the above arguments almost verbatim to establish the analogues of the theorems of the current section in this situation over the same two classes of fields, i.e. over 2-dimensional global fields $K$ of characteristic $\neq 2$ equipped with a divisorial set of places $V$ and over the function fields $K = k(C)$ of smooth geometrically integral curves $C$ over a field $k$ of characteristic $\neq 2$ satisfying condition ${\rm (F'_2)}$ equipped with the set $V$ of geometric places associated with the closed points of $C$. Details will be published elsewhere. We now recall that for both classes, ${}_2\Br(K)_V$ is known to be finite, implying that the number of isomorphism classes of central quaternion division $K$-algebras $D$ such that there exists a nondegenerate $n$-dimensional hermitian form $h$ for which $G = \mathrm{SU}_n(D , h)$ has good reduction at all $v \in V$ is finite. Eventually, this shows that the set of $K$-isomorphism classes of absolutely almost simple simply connected $K$-groups of type $\textsf{C}_n$ that split over a quadratic extension of the base field and have good reduction at all $v \in V$ is finite (for $K$ and $V$ as above). Observe, that Remark 8.2 yields a similar statement for type $\textsf{A}_n$, and the results of \ref{S:G2} do so for groups of type $\textsf{G}_2$. So, it would be interesting to see if this statement in fact extends to forms of all types that split over a quadratic extension of the base field --- this would be an important test for the general problem $(*)$ from \S\ref{S:Intro}.

\vskip5mm

\section{Groups of type $\mathsf{G_2}$}\label{S:G2}

\vskip3mm

Let $G_0$ be the split group of type
$\textsf{G}_2$ over a field $K$ of characteristic $\neq 2$. Then the $K$-isomorphism classes of $K$-groups of type $\textsf{G}_2$ are in a natural one-to-one
correspondence with the elements of the (pointed) set $H^1(K , G_0)$. Furthermore, there is a natural map
$$
\lambda_K \colon H^1(K , G_0) \to H^3(K , \mu_2)
$$
that has the following explicit description: if $\xi \in H^1(K , G_0)$ and the twisted group $G = {}_{\xi} G_0$
is the automorphism group of the octonion algebra $\mathbb{O} = \mathbb{O}(a , b , c)$ corresponding to a triple
$(a, b, c) \in {(K^{\times})}^3$, then
$$
\lambda_K(\xi) = \chi_a \cup \chi_b \cup \chi_c,
$$
where for $t \in K^{\times}$ we let $\chi_t \in H^1(K , \mu_2)$ denote, by abuse of notation, the cohomology class of the usual 1-cocycle $\chi_t$ coming from Kummer theory.
%represented by the cocycle/character
%$$
%\chi_t(\sigma) = \sigma\left( \sqrt{t} \right) / \sqrt{t} \ \ \text{for} \ \ \sigma \in \Ga(\bar{K}/K).
%$$
It is well-known that $\lambda_K$ is injective (cf. \cite[Ch. III, Appendix 2, 3.3]{Serre-GC}).

Now, suppose that $K$ is equipped with a discrete valuation $v$ such that $\mathrm{char}\: K^{(v)} \neq 2$. Then the automorphism group $G$ of an octonian algebra $\mathbb{O}$ has good reduction at $v$ if and only if $\mathbb{O}$ can be represented by a triple $(a, b, c) \in (K^{\times})^3$ such that
$$
v(a) = v(b) = v(c) = 0.
$$
It follows that if $G = {}_{\xi} G_0$ has good reduction at $v$, then the cocycle $\lambda_K(\xi) \in H^3(K , \mu_2)$ is unramified at $v$. Now,
using Corollary \ref{C:mu2}, we obtain the following.
\begin{thm}\label{T:G2-1}
Let $K$ be a 2-dimensional global field of characteristic $\neq 2$, $V$ a divisorial set of places, and $G$ a simple algebraic $K$-group of type $\textsf{G}_2$.

\vskip2mm

\noindent \hskip6pt  {\rm (i)} \parbox[t]{15.5cm}{The number of $K$-isomorphism classes of $K$-forms $G'$ having good reduction at all $v \in V$ is finite.}

\vskip2mm

\noindent \hskip3pt {\rm (ii)} \parbox[t]{15.5cm}{The genus $\gen_K(G)$ is finite.}

\vskip2mm

\noindent {\rm (iii)} \parbox[t]{15.5cm}{The map $\displaystyle \theta_{G , V} \colon H^1(K , G) \longrightarrow \prod_{v \in V} H^1(K_v , G)$ is proper.}
\end{thm}
\begin{proof}
(i) immediately follows from the remarks preceding the statement of the theorem and Corollary \ref{C:mu2}. (ii) is derived from (i) and Theorem \ref{T:Genus-smooth} just like Theorem \ref{T:genus} was derived from Theorem \ref{T:F1}. Finally, (iii) is derived from (i) and the observation mentioned in \S5.2. (Alternatively, one can use the injectivity of $\lambda_K$ and the fact that the kernel $\Omega_3$ of the map
$$
H^3(K , \mu_2) \longrightarrow \prod_{v \in V} H^3(K_v , \mu_2)
$$
is finite, which immediately follows from the finiteness of $H^3(K , \mu_2)_V$.)
\end{proof}

\vskip2mm

Next, we describe the analogues for groups of type $\textsf{G}_2$ of the results of \S \ref{S:Fm}.
\begin{thm}\label{T:G2-2}
Let $k$ be a field of characteristic $\neq 2$  that satisfies condition $(\mathrm{F}'_2)$, let $K = k(C)$ be the function field of a smooth affine geometrically integral curve $C$ over $k$, and let $V_0$ be the set of places of $K$ corresponding to the closed points of $C$. If $G$ is a simple algebraic $K$-group of type $\textsf{G}_2$, then

\vskip2mm

\noindent \hskip3mm {\rm (i)} \parbox[t]{15.5cm}{The number of $K$-isomorphism classes of $K$-forms $G'$ having good reduction at all $v \in V_0$ is finite.}

\vskip2mm

\noindent {\rm (ii)}  \parbox[t]{15.5cm}{The map $\displaystyle \theta_{G , V_0} \colon H^1(K , G) \longrightarrow \prod_{v \in V_0} H^1(K_v , G)$ is proper.}
\end{thm}
\begin{proof}
Again, (i) follows from the discussion preceding the statement of Theorem \ref{T:G2-1} and Theorem \ref{T:IRap}. Then (ii) follows from (i) and Remark 4.9. (As in the proof of Theorem \ref{T:G2-1}, one can alternatively use the injectivity of $\lambda_K$ in conjunction with the fact that the homomorphism
$$
H^3(K , \mu_2) \longrightarrow \prod_{v \in V_0} H^3(K_v , \mu_2)
$$
has finite kernel, which is a consequence of the finiteness of $H^3(K , \mu_2)_{V_0}$.)
\end{proof}

\vskip3mm

We now turn to a result that provides information about the genus of a group of type $\textsf{G}_2$ over the fields of rational functions over global fields.
\begin{thm}\label{T:G2-3}
Let $K = k(x_1, \ldots , x_r)$ the field of rational functions in $r$ variables over a global field $k$ of characteristic $\neq 2$, and let $G$ be a simple $K$-group of type $\textsf{G}_2$.

\vskip2mm

\noindent \hskip3pt {\rm (i)} \parbox[t]{15.5cm}{If $r = 1$, then the genus $\gen_K(G)$ reduces to a single element.}

\vskip2mm

\noindent {\rm (ii)} \parbox[t]{15.5cm}{The genus $\gen_K(G)$ is finite for any $r$.}
\end{thm}

\vskip2mm

The proof requires more detailed information about the ``residues" of groups of type $\textsf{G}_2$ at places of bad reduction. So, let $K$ be a field with a discrete valuation $v$ such that the residue field $\mathscr{K} = K^{(v)}$ is of characteristic $\neq 2$, and let
$$
\partial_v \colon H^3(K , \mu_2) \to H^2(\mathscr{K} , \mu_2)
$$
be the corresponding residue map. It is well-known (cf. \cite[\S\S 7.1 and 7.5]{Gille}) that any ``decomposable" element $\chi_a \cup \chi_b \cup \chi_c$, where
$a, b, c \in K^{\times}$ (and by the Bloch-Kato conjecture such elements generate $H^3(K , \mu_2)$) can be written in the form  $\chi_{a'} \cup \chi_{b'} \cup \chi_{c'}$ with $v(a') = v(b') = 0$ and $v(c') = 0$ or $1$, and that on elements of this form the residue map is given by:
$$
\partial_v(\chi_a \cup \chi_b \cup \chi_c) = \left\{
\begin{array}{lcl}
0 & , & v(a) = v(b) = v(c) = 0, \\
\chi_{\bar{a}} \cup \chi_{\bar{b}} & , & v(a) = v(b) = 0, \ v(c) =
1,
\end{array} \right.
$$
where $\bar{a} , \bar{b}$ are the images of $a , b$ in $\mathscr{K}^{\times}$. We define the quaternion algebra corresponding to the residue of such an element to be the matrix algebra $M_2(\mathscr{K})$ in the first case and the standard quaternion algebra $\displaystyle \left( \frac{\bar{a} , \bar{b}}{\mathscr{K}}  \right)$
in the second. We will assume henceforth that $\mathscr{K}$ is {\it finitely generated}.

\vskip2mm

\begin{lemma}\label{L:G2-1}
Given $\xi_1, \xi_2 \in H^1(K , G_0)$, we denote by $G_i = {}_{\xi_i}G_0$ the corresponding group and set $\mathscr{D}_i$ to be the quaternion algebra
over $\mathscr{K}$ corresponding to the residue $\partial_v(\lambda_K(\xi_i))$ for $i = 1, 2$.
If $G_1$ and $G_2$ have the same isomorphism classes of maximal $K$-tori, then $\mathscr{D}_1$ and $\mathscr{D}_2$
have the same isomorphism classes of maximal \'etale $\mathscr{K}$-subalgebras.
\end{lemma}
\begin{proof}
Let $\mathbb{O}_i$ be an octonian $K$-algebra such that $G_i = \mathrm{Aut}(\mathbb{O}_i)$ for $i = 1, 2$. Suppose $\mathbb{O}_i$ can
be represented by a triple $(a_i, b_i, c_i) \in (K^{\times})^3$ such that $v(a_i) = v(b_i) = v(c_i) = 0$. Then $G_i$ has a good reduction
at $v$. Since $G_1$ and $G_2$ have the same isomorphism classes of maximal $K$-tori, by Theorem \ref{T:Genus-smooth} the group $G_{3-i}$ also has good reduction at $v$. This means that $\mathbb{O}_{3-i}$ can also be represented by a triple $(a_{3-i}, b_{3-i}, c_{3-i}) \in (K^{\times})^3$ such that $v(a_{3-i}) = v(b_{3-i}) = v(c_{3-i}) = 0$. In this case, both residues $\partial_v(\lambda_K(\xi_1))$ and $\partial_v(\lambda_K(\xi_2))$ are trivial, so our claim is obvious.
It remains to consider the case where $\mathbb{O}_i$ is represented by a triple $(a_i, b_i, c_i) \in (K^{\times})^3$ satisfying
$$
v(a_i) = v(b_i) = 0, \ \ v(c_i) = 1
$$
for $i = 1, 2$. Then $\mathscr{D}_i$ is the quaternion algebra $\displaystyle \left( \frac{\bar{a}_i , \bar{b}_i}{\mathscr{K}} \right)$ where $\bar{a}_i , \bar{b}_i$ are the images of $a_i , b_i$ in $\mathscr{K}^{\times}$. To prove that $\mathscr{D}_1$ and $\mathscr{D}_2$ have the same maximal \'etale subalgebras, it suffices to prove the following:

\vskip2mm

\begin{center}

\parbox[t]{15cm}{\it Let $\mathscr{L}$ be an extension of $\mathscr{K}$ of degree $\leqslant 2$. Then $\mathscr{L}$ splits $\mathscr{D}_1$ if and only if
it splits $\mathscr{D}_2$.}

\end{center}

\vskip2mm

Let $L$ be the unramified extension of $K_v$ with the residue field $\mathscr{L}$. We have
$$
\partial_v(\lambda_K(\xi_i)) = \hat{\partial}_v(\lambda_{K_v}(\hat{\xi}_i)),
$$
where $\hat{\partial}_v \colon H^3(K_v , \mu_2) \to H^2(\mathscr{K} , \mu_2)$ is the residue map and $\hat{\xi}_i$ is the image of $\xi_i$ under the restriction map $H^1(K , G_0) \to H^1(K_v , G_0)$. Since residue maps commute with restriction maps for unramified extensions, we further obtain
\begin{equation}\label{E:L}
\mathrm{Res}_{\mathscr{L}/\mathscr{K}}(\partial_v(\lambda_K(\xi_i))) = \hat{\partial}^L_v(\lambda_{L}(\mathrm{Res}_{L/K_v}(\hat{\xi}_i))),
\end{equation}
where $\hat{\partial}^L_v \colon H^3(L , \mu_2) \to H^2(\mathscr{L} , \mu_2)$ is the residue map for $L$ (see \cite[Proposition 3.3.1]{CT-SB}).

Now, suppose that $\mathscr{L}$ splits $\mathscr{D}_i$, i.e. the quaternion algebra $\displaystyle \left( \frac{\bar{a}_i , \bar{b}_i}{\mathscr{L}} \right)$ is trivial. Then it follows from Hensel's Lemma that the quaternion algebra $\displaystyle \left( \frac{a_i , b_i}{L} \right)$ is also trivial, and therefore the group $G_i$ becomes split over $L$. Define a two-dimesional $K_v$-torus $T$ to be $$\mathbb{G}_m \times \mathbb{G}_m \ \ \text{if} \ \  L = K_v \ \ \ \  \text{and} \ \ \ \  \mathrm{R}^{(1)}_{L/K_v}(\mathbb{G}_m) \times \mathrm{R}^{(1)}_{L/K_v}(\mathbb{G}_m) \ \  \text{if} \ \ [L : K_v] = 2$$ (where, as usual, $\mathrm{R}^{(1)}_{L/K_v}(\mathbb{G}_m)$ denotes the corresponding norm torus). Then $T$ is isomorphic to a maximal $K_v$-torus of $G_i$. Since $G_1$ and $G_2$ have the same isomorphism classes of maximal $K$-tori, they also have the same isomorphism classes of maximal $K_v$-tori (see \cite[Remark 2.2]{RR}). Thus, $T$ is isomorphic to a maximal $K_v$-torus of $G_{3-i}$ as well. Then the cocycle $\mathrm{Res}_{L/K_v}(\hat{\xi}_{3-i})$ is trivial, so applying (\ref{E:L}) with $i$ replaced by $3-i$, we obtain that
$\mathrm{Res}_{\mathscr{L}/\mathscr{K}}(\partial_v(\lambda_K(\xi_{3-i})))$ is trivial, i.e. $\mathscr{L}$ splits $\mathscr{L}_{3-i}$, as required.
\end{proof}

\vskip0.2mm

\begin{cor}\label{C:G2}
Assume that the residue field $\mathscr{K} = K^{(v)}$ is finitely generated and has the following property:

\vskip2mm

\noindent $(\star)$ \parbox[t]{15cm}{If $D_1$ and $D_2$ are central quaternion
algebras over $\mathscr{K}$ that have the same maximal \'etale subalgebras,
then $D_1 \simeq D_2$.}

\vskip2mm

\noindent If $\xi_1 , \xi_2 \in H^1(K , G_0)$ are such that the corresponding groups $G_i = {}_{\xi_i} G_0$ for $i = 1, 2$ have the same isomorphism
classes of maximal $K$-tori, then $\partial_v (\lambda_K(\xi_1)) = \partial_v (\lambda_K(\xi_2))$.
\end{cor}

\vskip3mm

\noindent {\it Proof of Theorem \ref{T:G2-3}(i).} We will prove a more general result. To formulate it, in addition to property $(\star)$ of a field $\mathscr{K}$  from Corollary \ref{C:G2}, we need to introduce the following property of a field $k$:

\vskip2mm

\noindent $(\star\star)$ \ \parbox[t]{15cm}{\baselineskip=5mm If $G_1$ and $G_2$ are two $k$-groups of type $\textsf{G}_2$ having the same maximal tori, then $$G_1 \simeq G_2 \ \ \text{over} \ \ k. \hskip3cm $$}

\vskip2mm

\begin{thm}\label{T:G2-5}
Assume that a field $k$ of characteristic $\neq 2$ satisfies $(\star\star)$ and that any finite extension $\ell$ of $k$ satisfies $(\star)$. Then the field of rational functions $K = k(x)$ satisfies $(\star\star)$.
\end{thm}

We note that the mere fact that $k$ satisfies $(\star)$ implies that $K = k(x)$ also satisfies $(\star)$ - see \cite[Theorem A]{RR}. Furthermore, it is well-known that global fields satisfy both conditions $(\star)$ (cf. \cite[3.6]{CRR1}) and $(\star\star)$ (cf. \cite[Theorem 7.5]{PrRap-WC}), so Theorem \ref{T:G2-5} immediately yields
the assertion of Theorem \ref{T:G2-3}(i).

\vskip2mm

To prove Theorem \ref{T:G2-5}, we let $V$ denote the set of all places of $K = k(x)$ that are trivial on $k$ (these of course correspond to the closed
points of $\mathbb{P}^1_k$). It follows from the above remarks that we may assume the field $k$ to be infinite. Now, let $G' \in \gen_K(G)$, and let $\xi, \xi' \in H^1(K , G_0)$ be the corresponding cocycles (so that $G = {}_{\xi}G_0$ and $G' = {}_{\xi'}G_0$). For any $v \in V$, since the residue field $K^{(v)}$ is a finite extension of $k$, hence satisfies $(\star)$ by our assumption, we conclude from Corollary \ref{C:G2} that
$$
\partial_v(\lambda_K(\xi)) = \partial_v(\lambda_K(\xi')).
$$
Thus,
\begin{equation}\label{E:G2-007}
\lambda_K(\xi') = \lambda_K(\xi) \cdot \zeta \ \ \text{for some} \ \ \zeta \in H^3(K , \mu_2)_V.
\end{equation}
But according to Faddeev's sequence (cf. \cite[Ch. III, \S 9]{GMS}, \cite[6.9]{Gille}), the natural (inflation) map $\iota \colon H^3(k , \mu_2) \to H^3(K , \mu_2)$ identifies $H^3(k , \mu_2)$ with
the unramified group $H^3(K , \mu_2)_V$. Then in terms of this identification, $\zeta$ in (\ref{E:G2-007}) belongs to $H^3(k , \mu_2)$.  We would like to show that actually $\zeta = 1$, which will imply that $G' \simeq G$.

Since $k$ is infinite, we can pick $v \in V$ so that $K^{(v)} = k$ and $G$ and $G'$ have smooth reductions $\uG^{(v)}$ and ${\uG'}^{(v)}$ at $v$. We have the natural map $$\epsilon_v \colon H^3(K , \mu_2)_{\{ v \}} \to H^3(k , \mu_2)$$ (defined on the $v$-unramified part) such that $\epsilon_v \circ \iota = \mathrm{id}$. Clearly, the images under $\lambda_k$ of the cocycles that correspond to $\uG^{(v)}$ and ${\uG'}^{(v)}$ coincide with $\epsilon_v(\lambda_K(\xi))$ and $\epsilon_v(\lambda_K(\xi'))$, respectively. Now, by the last assertion of Theorem \ref{T:Genus-smooth}, the groups $\uG^{(v)}$ and ${\uG'}^{(v)}$ have the same isomorphism classes of maximal $k$-tori, so since $k$ is assumed to satisfy $(\star\star)$, we conclude that $\uG^{(v)} \simeq {\uG'}^{(v)}$. This means that
\begin{equation}\label{E:G2-008}
\epsilon_v(\lambda_K(\xi)) = \epsilon_v(\lambda_K(\xi')).
\end{equation}
Applying $\epsilon_v$ to (\ref{E:G2-007}) and comparing the result with (\ref{E:G2-008}), we obtain that $\epsilon_v(\zeta) = 1$, and consequently $\zeta = 1$, as required. \hfill $\Box$

\vskip2mm

\noindent {\it Proof of Theorem \ref{T:G2-3}(ii).} We view $K$ as the field of rational functions on $\mathbb{P}_k^r$, and let $V$ be the set of discrete valuations on $K$ associated with the prime divisors. Pick a finite subset $V_0 = \{ v_1, \ldots , v_d \}$ of $V$ so that $G$ has a good reduction at all $v \in V \setminus V_0$. For each $i = 1, \ldots , d$, we let $\mathscr{D}_i$ denote the quaternion algebra over the residue field $\mathscr{K}_i = K^{(v_i)}$ that corresponds to the residue $\partial_{v_i}(\lambda_K(\xi))$, where $\xi \in H^1(K , G_0)$ is such that $G = {}_{\xi}G_0$. Given any other $K$-group $G'$ of
type $\textsf{G}_2$, we set $\xi' = \xi(G')$ to be the cohomology class in $H^1(K , G_0)$ such that $G' = {}_{\xi'} G_0$, and for $i = 1, \ldots , d,$ denote by $\mathscr{D}'_i = \mathscr{D}_i(G')$ the quaternion algebra corresponding to the residue $\partial_{v_i}(\lambda_K(\xi'))$. Then Lemma \ref{L:G2-1} tells us that the correspondence $G' \mapsto (\mathscr{D}'_i)$ gives a map
$$
\rho \colon \gen_K(G) \longrightarrow \prod_{i = 1}^r \gen_{\mathscr{K}_i}(\mathscr{D}_i),
$$
where the genus $\gen_{\mathscr{K}}(\mathscr{D})$ of a central quaternion algebra $\mathscr{D}$ over a field $\mathscr{K}$ (of characteristic $\neq 2$) is defined to be the set of the Brauer classes $[\mathscr{D'}]$ of central quaternion $\mathscr{K}$-algebras $\mathscr{D}'$ having the same maximal \'etale subalgebras as $\mathscr{D}$ -- see \cite{CRR2}, \cite{CRR3} for the details. Since each genus $\gen_{\mathscr{K}_i}(\mathscr{D}_i)$ is finite \cite{CRR2}, in order to prove that $\gen_K(G)$ is finite, it suffices to show that each fiber $\rho^{-1}(\rho(G'))$ for $G' \in \gen_K(G)$ is finite. For this we note that Theorem \ref{T:Genus-smooth} implies that for any $G' \in \gen_K(G)$ and any $v \in V \setminus V_0$, the residue $\partial_v(\lambda_K(G'))$ is trivial. Consequently, for any $G' \in \gen_K(G)$, any $G'' \in \rho^{-1}(\rho(G'))$, and corresponding cocycles $\xi' , \xi''$, we have
$$
\lambda_K(\xi'') = \lambda_K(\xi') \cdot \zeta \ \ \text{for some} \ \ \zeta \in H^3(K , \mu_2)_V.
$$
But from Faddeev's sequence, we again see that in our situation $H^3(K , \mu_2)_V = H^3(k , \mu_2)$ (cf. \cite[Theorem 10.1]{GMS}). Since $k$ is a global field, the latter group is finite by Poitou-Tate, so, from the injectivity of $\lambda_K$, we obtain the finiteness of the fibers $\rho^{-1}(\rho(G'))$ for $G' \in \gen_K(G)$, and hence the finiteness of $\gen_K(G)$. \hfill $\Box$

\vskip5mm

\section{Appendix: An alternative proof of Theorem \ref{T:H3-Finite} in characteristic zero}\label{S:Append}

\vskip3mm

Let $K$ be a 2-dimensional global field of characteristic zero. Take any smooth geometrically integral affine curve $C$ over a number field $k$ such that
$K = k(C)$, and let $V_0$ denote the set of discrete valuations of $K$ corresponding to the closed points points of $C$. Furthermore, we can pick a finite subset $S \subset V^k$ containing all archimedean places so that there exists a model $\mathcal{C}$ of $C$ over the ring of $S$-integers $\mathcal{O}_{k,S}$ that has good reduction at all $v \in V^k \setminus S$. Then every such $v$ has a canonical extension $\tilde{v}$ to $K$ defined by $\mathcal{C}$, and we let $V_1 = \{\,\tilde{v} \, \vert \, v \in V^k \setminus S\,\}$. It is easy to see that any divisorial set of places of $K$ contains $V_0 \cup V_1$ for a suitable choice of $C$, $S$ and $\mathcal{C}$. So, it is enough to prove the following.
\begin{thm}\label{T:char0}
Set $V = V_0 \cup V_1$ in the above notations. Let $n \geqslant 1$ be an integer such that $S$ contains all divisors of $n$, so that
$\mathrm{char}\: K^{(v)}$ is prime to $n$ for all $v \in V$. Then the unramified cohomology group $H^3(K , \mu^{\otimes 2}_n)_V$ is finite.
\end{thm}

\vskip2mm

We begin by developing some formalism that applies in a more general situation.

\vskip2mm

\noindent {\bf 1. Two injectivity results.} Let $C$ be a smooth geometrically integral curve over an arbitrary field $k$ and let $p$ be a prime $\neq \mathrm{char}\: k$. Then for any $n \geqslant m$,
we have the following commutative diagram of \'etale sheaves on $C$
$$
\xymatrix{1 \ar[r] & \mu_{p^m} \ar[d]^{{\rm id}} \ar[r] & \mathbb{G}_m \ar[d]^{{\rm id}} \ar[r]^{[p^m]} & \mathbb{G}_m \ar[d]^{[p^{n-m}]} \ar[r] & 1 \\ 1 \ar[r] & \mu_{p^n} \ar[r] & \mathbb{G}_m \ar[r]^{[p^n]} & \mathbb{G}_m \ar[r] & 1}
$$
%$$
%\begin{array}{cccccclcc}
%1 & \to & \mu_{p^m} & \longrightarrow & \mathbb{G}_m & \stackrel{[p^m]}{\longrightarrow} & \mathbb{G}_m & \to & 1 \\
%  &     & \downarrow \mathrm{id} &            &  \downarrow \mathrm{id} &                      &  \downarrow [p^{n-m}] &  &   \\
%1 & \to & \mu_{p^n} & \longrightarrow & \mathbb{G}_m & \stackrel{[p^n]}{\longrightarrow} & \mathbb{G}_m & \to & 1
%\end{array},
%$$
where $[p^{\ell}]$ denotes the morphism $x \mapsto x^{p^{\ell}}$. Passing to cohomology, we get the following commutative diagram
\begin{equation}\label{E:XX001}
\xymatrix{ {\rm Pic}(C) \ar[d]^{{\rm id}} \ar[r]^{[p^m]} & {\rm Pic}(C) \ar[d]^{[p^{n-m}]} \ar[r] & H^2(C, \mu_{p^m}) \ar[r] \ar[d] & \Br(C) \ar[d]^{{\rm id}} \ar[r]^{[p^m]} & \Br(C) \ar[d]^{[p^{n-m}]} \\ {\rm Pic}(C) \ar[r]^{[p^n]} & {\rm Pic}(C) \ar[r] & H^2(C, \mu_{p^n}) \ar[r] & \Br(C) \ar[r]^{[p^n]} & \Br(C)}
\end{equation}
where $[p^{\ell}]$ denotes multiplication by $p^{\ell}$. We now use the following elementary statement.
\begin{lemma}\label{L:XX001}
Let $A$ be an abelian group, and $p$ be a prime. Consider the family of abelian groups $A/p^nA$ for $n \geqslant 1$ with morphisms $\pi^m_n \colon A/p^mA \to A/p^nA$ for $n \geqslant m$ given by $$\pi^m_n(a + p^mA) = p^{n-m} a + p^nA.$$ Then the direct limit $\displaystyle \lim_{\longrightarrow} (A/p^nA , \pi^m_n)$
can be naturally identified with $A \otimes_{\Z} \Q_p/\Z_p$.
\end{lemma}
\begin{proof}
For $n \geqslant 1$, define $\lambda_n \colon A/p^nA \to A \otimes_{\Z} \Q_p/\Z_p$ by $a + p^nA \mapsto a \otimes (p^{-n} + \Z_p)$. Since
$$
p^na \otimes (p^{-n} + \Z_p) = a \otimes (1 + \Z_p) = 0 \ \ \text{in} \ \ A \otimes_{\Z} \Q_p/\Z_p,
$$
this map is well-defined. Furthermore, for $n \geqslant m$ we have
$$
\lambda_m(a + p^mA) = a \otimes (p^{-m} + \Z_p) = p^{n-m}a \otimes (p^{-n} + \Z_p) = \lambda_n(\pi^m_n(a + p^mA)),
$$
so the $\lambda_m$'s assemble into a (surjective) homomorphism
$$
\lambda \colon \lim_{\longrightarrow} (A/p^nA , \pi^m_n) \longrightarrow A \otimes_{\Z} \Q_p/\Z_p.
$$
To construct the inverse map, we start with a map
$$
 A \times \Q_p/\Z_p \longrightarrow \lim_{\longrightarrow} (A/p^nA , \pi^m_n), \ \ \ (a , p^{-n}b + \Z_p) \mapsto ba + p^nA \ \ \text{(where} \ b \in \Z , \ n \geqslant 1),
$$
and check that this map is well-defined and bilinear. This yields a homomorphism
$$
A \otimes_{\Z} \Q_p/\Z_p \longrightarrow \lim_{\longrightarrow} (A/p^nA , \pi^m_n)
$$
which is easily seen to be the inverse of $\lambda$.
\end{proof}

\vskip3mm

As usual, for an abelian group $A$ and $n \in \mathbb{N}$, we write ${}_nA = \{ a \in A \: \vert \: n a = 0 \}$. Furthermore, for a prime $p$, we write ${}_{p^{\infty}} A = \{a \in A \: \vert \: p^n a = 0 \ \ \text{for some} \ \ n \geqslant 1 \}$. If $p \neq \mathrm{char}\: k$, then we set
$$
\Q_p/\Z_p(d) = \lim_{\longrightarrow} \mu_{p^n}^{\otimes d}.
$$
Taking the direct limits of the diagrams (\ref{E:XX001}) over all $n \geqslant m$ and  using Lemma \ref{L:XX001}, we obtain the following exact sequence
\begin{equation}\label{E:XX003}
0 \to \Pic(C) \otimes_{\Z} \Q_p/\Z_p \longrightarrow H^2(C , \Q_p/\Z_p(1)) \longrightarrow {}_{p^{\infty}}\Br(C) \to 0,
\end{equation}
which leads to our first injectivity result.
\begin{lemma}\label{L:XX02}
For a prime $p \neq \mathrm{char}\: k$, if $\Pic(C) \otimes_{\Z} \Q_p/\Z_p = 0$ (in particular, if $\Pic(C)$ is torsion), then the canonical map $H^2(C , \Q_p/\Z_p(1)) \longrightarrow {}_{p^{\infty}} \Br(C)$ is an isomorphism.
\end{lemma}

\vskip1mm

In particular, if $k$ is a finite field and $C$ is affine, then $\Pic(C)$ is finite, and we obtain the following.
\begin{cor}\label{C:XX001}
If $k$ is a finite field and $C$ is affine, then for any prime $p \neq \mathrm{char}\: k$,  the canonical map $$H^2(C , \Q_p/\Z_p(1)) \longrightarrow \Br(C)$$ is injective.
\end{cor}

\vskip2mm

The second injectivity result that we need is the following well-known consequence of the truth of the Bloch-Kato conjecture (cf. \cite[p. 3]{Jann2}).
\begin{lemma}\label{L:XX002}
Let $p$ be a prime $\neq \mathrm{char}\: k$. Then for all $n \geqslant 1$,  the map
$$
H^{\ell}(k , \mu_{p^n}^{\otimes (\ell-1)}) \longrightarrow H^{\ell}(k , \Q_p/\Z_p(\ell-1))
$$
induced by the natural embedding $\mu_{p^n}^{\otimes(\ell-1)} \hookrightarrow \mathbb{Q}_p/
\mathbb{Z}_p (\ell-1)$ is injective for any $\ell \geqslant 2$.
\end{lemma}

\vskip3mm

\noindent {\bf 2. The fundamental sequence and unramified cohomology.} Let $k$ be a field equipped with a discrete valuation $v$, with valuation ring $\mathcal{O}_v$ and residue field $k^{(v)}.$ Fix a prime $p \neq \mathrm{char}\: k^{(v)}$. Given a smooth geometrically integral affine curve $C$ over $k$, we let $\widetilde{C}$ denote the smooth geometrically integral projective curve over $k$ that contains $C$ as an open subset. We will assume there exist models $\mathcal{C} \subset \widetilde{\mathcal{C}}$ of these curves over $\mathcal{O}_v$ such that the associated reductions $\uC^{(v)} \subset \underline{\widetilde{C}}^{(v)}$ are smooth, geometrically integral and satisfy
$$
\vert \widetilde{C}(\bar{k}) \setminus C(\bar{k}) \vert = \vert \underline{\widetilde{C}}^{(v)}(\overline{k^{(v)}}) \setminus \uC^{(v)}(\overline{k^{(v)}}) \vert.
$$
Then the specialization map defines an isomorphism of the maximal pro-$p$ quotients of the fundamental groups:
\begin{equation}\label{E:FundGr}
\pi_1(C \otimes_k \bar{k})^{(p)} \longrightarrow \pi_1(\uC^{(v)} \otimes_{k^{(v)}} \overline{k^{(v)}})^{(p)}
\end{equation}
(with a compatible choice of base points) --- see \cite[Ch. XIII]{SGA1}.
%Suppose $C$ is a smooth geometrically integral affine curve over $k$ that has a model $\mathcal{C}$ over $\mathcal{O}_v$ such that the associated reduction
%Let $C$ be an affine curve over $k$, let $\mathcal{C}$ be a model of $C$ over the valuation ring $\mathcal{O}_v \subset k$, and assume that the reduction
%$\uC^{(v)}$ is smooth and geometrically integral.

As in our previous discussion, let $\tilde{v}$ be the extension of $v$ to $K = k(C)$ defined by $\mathcal{C}$. We say that an element of $H^{\ell}(C , \mu_{p^m}^{\otimes d})$ is \emph{unramified} at $v$ if its image in $H^{\ell}(K , \mu_{p^m}^{\otimes d})$ is unramified at $\tilde{v}$ in the usual sense.

Now, set
$$
M_m(p) = H^1(C \otimes_k \bar{k} , \mu_{p^m}^{\otimes 2}) \ \ \text{for} \ \ m \geqslant 1 \ \ \text{and} \ \ M(p) = \lim_{\longrightarrow} M_m(p) = H^1(C \otimes_k \bar{k} , \Q_p/\Z_p(2)).
$$
Considering the fundamental sequences (\ref{E:X0}) in \S\ref{S:ProofTF1} for $d = 2$, $\ell = 3$ and $n = p^m$ $(m \geqslant 1)$ and taking their direct limit, we obtain a map
$$
H^3(C , \Q_p/\Z_p(2)) \stackrel{\omega^{2,3}_k(p)}{\longrightarrow} H^2(k , M(p)).
$$
Similarly, we obtain a map
$$
H^2(\uC^{(v)} , \Q_p/\Z_p(1)) \stackrel{\omega^{1,2}_{k^{(v)}}(p)}{\longrightarrow} H^1(k^{(v)} , M^{(v)}(p)) \ \ \text{where} \ \  M^{(v)}(p) = H^1(\uC^{(v)} \otimes_{k^{(v)}} \overline{k^{(v)}} , \Q_p/\Z_p(1)).$$
The isomorphism (\ref{E:FundGr}) enables us to identify $$M^{(v)}(p) = \mathrm{Hom}(\pi_1(\uC^{(v)} \otimes_{k^{(v)}} \overline{k^{(v)}}) , \Q_p/\Z_p(1))$$ with
$\mathrm{Hom}(\pi_1(C \otimes_k \bar{k}) , \Q_p/\Z_p(1))$, which in turn, according to (\ref{E-FundGp}), can be identified with the twist $M(p)(-1)$. In the sequel (particularly, in  the proof of Proposition \ref{P:unramX1}), we will routinely use the identification of $M(p)(-1)$ with $M^{(v)}(p)$ as Galois modules compatible with the canonical identification of the decomposition group of $v$ with the absolute Galois group of $k^{(v)}$.
Furthermore,
in view of the isomorphism (\ref{E:FundGr}), the inertia group of $v$ acts trivially on $M_m(p)$, hence on on $M(p)$. We thus have the residue map
$$
\partial^{M(p)}_v \colon H^2(k , M(p)) \longrightarrow H^1(k^{(v)} , M(p)(-1))
$$
(obtained by taking the direct limit of the residue maps for all $M_m(p)$).
Our goal in this subsection is to prove that $\omega^{2 , 3}_k(p)$ takes unramified classes to unramified ones. More precisely, we have the following.
\begin{prop}\label{P:unramX1}
Let $p$ be a prime $\neq \mathrm{char}\: k^{(v)}$ and assume that $\Pic(\uC^{(v)}) \otimes_{\Z} \Q_p/\Z_p = 0$. If $x \in H^3(C , \Q_p/\Z_p(2))$ is unramified at $v$ in the sense specified above, then $\omega^{2 , 3}_k(p)(x)$ is also unramified at $v$ (i.e., $\partial_{v}^{M(p)}(\omega^{2 , 3}_k(p)(x)) = 0$).
\end{prop}

The proof critically depends on the existence of certain analogues of residue maps with nice properties in the \'etale cohomology of curves. Namely, let $n$ be an integer that is invertible in $\mathcal{O}_v$. Combining the localization sequence with absolute purity,
%Composing the localization map followed by the purity isomorphism,
one obtains the following ``residue map":
$$
\rho_v^{\ell} \colon H^{\ell}(C, \mu_n^{\otimes b}) \to H^{\ell-1}(\uC^{(v)}, \mu_n^{\otimes(b-1)})
$$
(see \cite[\S 3.2]{CT-SB} and \cite{IRap-Residue} for the details). It should be noted that this approach in fact enables one to recover the usual residue maps in Galois cohomology, at least up to sign (cf. \cite{JannSS}). We need the following properties of the above maps.

\begin{thm}\label{T:Residue}
{\rm (\cite{IRap-Residue})} For every $\ell \geq 2$, we have commutative diagrams
\begin{equation}\tag{I}
\xymatrix{
H^{\ell}(C, \mu_n^{\otimes b}) \ar[rr]^{\omega_k^{b, \ell}} \ar[d]_{\rho_v^{\ell}} &  & H^{\ell-1}(k, H^1(C \otimes_k \overline{k} , \mu_n^{\otimes b})) \ar[d]^{\partial_v^{\ell-1}} \\ H^{\ell-1}(\uC^{(v)}, \mu_n^{\otimes(b-1)}) \ar[rr]^{\omega_{k^{(v)}}^{b-1, \ell-1}} & & H^{\ell-2}(k^{(v)}, H^1 (\uC^{(v)} \otimes_{k^{(v)}} \overline{k^{(v)}} , \mu_n^{\otimes (b-1)}))}
\end{equation}
and
\begin{equation}\tag{II}
\xymatrix{
H^{\ell}(C, \mu_n^{\otimes b}) \ar[r]^{\nu^{b, \ell}_k} \ar[d]_{\rho_v^{\ell}} & H^{\ell} (k(C), \mu_n^{\otimes b}) \ar[d]^{\delta_v^{\ell}} \\ H^{\ell-1} (\uC^{(v)}, \mu_n^{\otimes(b-1)}) \ar[r]^{\nu^{b-1, \ell-1}_{k^{(v)}}} & H^{\ell-1} (k^{(v)}(\uC^{(v)}), \mu_n^{\otimes(b-1)})}
\end{equation}
where $\nu^{b, \ell}_k$ and $\nu^{b-1, \ell-1}_{k^{(v)}}$ are the natural maps induced by passage to the generic point,
%(``evaluation at the generic point"),
and the maps $\partial_v^{\ell-1}$ and $\delta_v^{\ell}$ coincide up to sign with the usual residue maps in Galois cohomology.
\end{thm}

\vskip2mm

\noindent {\it Proof of Proposition \ref{P:unramX1}.}
Taking the direct limits of diagrams (I) and (II) above, we obtain the following commutative diagrams
\begin{equation}\label{E:XX1}
\xymatrix{H^3(C, \Q_p/\Z_p(2)) \ar[r]^{\omega^{2,3}_{k}(p)} \ar[d]_{\rho_{v}(p)} & H^2(k, M(p)) \ar[d]^{\partial_{v}^{2}(p)} \\ H^{2} (\uC^{(v)} , \Q_p/\Z_p(1)) \ar[r]^{\omega^{1,2}_{k^{(v)}}(p)} & H^{1}(k^{(v)} , M(p)(-1))}
\end{equation}
(here we use the identification of $M^{(v)}(p)$ with $M(p)(-1)$ mentioned earlier),
and
\begin{equation}\label{E:XX2}
\xymatrix{H^{3}(C , \Q_p/\Z_p(2)) \ar[r]^{\nu^2_{k}(p)} \ar[d]_{\rho_v(p)} & H^{3}(k(C) , \Q_p/\Z_p(2)) \ar[d]^{\delta^3_v(p)} \\ H^{2}(\uC^{(v)} , \Q_p/\Z_p(1)) \ar[r]^{\nu^1_{k^{(v)}}(p)} & H^{2}(k^{(v)}(\uC^{(v)}) , \Q_p/\Z_p(1))}
\end{equation}
where, up to sign, $\partial_v^2(p)$ coincides with with $\partial_v^{M(p)}$, and $\delta^3_v(p)$ with the residue map in Galois cohomology.
Since $\uC^{(v)}$ is smooth (so that $\Br(\uC^{(v)})$ injects into $\Br(k^{(v)}(\uC^{(v)})$), Lemma \ref{L:XX02} and our assumption that $\Pic(\uC^{(v)}) \otimes_{\Z} \Q_p/\Z_p = 0$ imply that $\nu^1_{k^{(v)}}(p)$ is injective. On the other hand, since $x$ is unramified at $v$,
using the commutativity of (\ref{E:XX2}), we obtain
$$
\delta^3_v(p)(\nu^2_{k}(p)(x)) = 0 = \nu^1_{k^{(v)}}(p)(\rho_v(p)(x)).
$$
So, we conclude that $\rho_v(p)(x) = 0$. The commutativity of (\ref{E:XX1}) then implies
$$
\partial^{2}_v(p)(\omega^{2,3}_k(p)(x)) = \omega^{1,2}_{k^{(v)}}(p)(\rho_v(p)(x)) = 0,
$$
and the required fact follows. \hfill $\Box$

\vskip2mm

\noindent {\bf 3. The unramified cohomology of $M(p)$.} Let now $k$ be a number field and $V \subset V^k_f$ be a cofinite set of (finite) places that does not contain any places lying above $p$ and such that $C$ has good reduction at all $v \in V$. Then, as we discussed above, for every $v \in V$ one has the residue map
$$
\partial^{M(p)}_v \colon H^2(k , M(p)) \longrightarrow H^1(k^{(v)} , M(p)(-1)).
$$
We then define the (second) unramified cohomology of $M(p)$ by
$$
H^2(k , M(p))_V = \{ x \in H^2(k , M(p)) \: \vert \: \partial^{M(p)}_v(x) = 0 \ \ \text{for all} \ \ v \in V \}.
$$
The following proposition is crucial for the proof of Theorem \ref{T:char0}.
\begin{prop}\label{P:XX002}
For any $m \geqslant 1$, the group ${}_{p^m} H^2(k , M(p))_V$ is finite.
\end{prop}

We begin with the following elementary statement.
\begin{lemma}\label{L:XX007}
Let $T$ be a torus over an arbitrary field $k$ and $p$ be a prime $\neq \mathrm{char}\: k$. Then

\vskip1mm

\noindent {\rm (1)} for any $i \geqslant 2$, the natural map $H^{i}(k , {}_{p^{\infty}}T(\bar{k})) \to {}_{p^{\infty}}H^{i}(k , T)$ is an isomorphism;

\vskip1mm

\noindent {\rm (2)} \parbox[t]{16cm}{for $i = 1$, we have an exact sequence
$$
0 \to T(k) \otimes_{\Z} \Q_p/\Z_p \longrightarrow H^1(k , {}_{p^{\infty}}T(\bar{k})) \longrightarrow {}_{p^{\infty}} H^1(k , T) \to 0. $$}
\end{lemma}
\begin{proof}
Let $T(p) = {}_{p^{\infty}}T(\bar{k})$. The exact sequence
$$
1 \to T(p) \longrightarrow T(\bar{k}) \longrightarrow T(\bar{k})/T(p) \to 1
$$
for any $i \geqslant 1$ gives rise to the following exact sequence
$$
{}_{p^{\infty}} H^{i-1}(k , T) \longrightarrow {}_{p^{\infty}} H^{i-1}(k , T(\bar{k})/T(p)) \longrightarrow {}_{p^{\infty}} H^{i}(k , T(p)) \longrightarrow {}_{p^{\infty}} H^{i}(k , T) \longrightarrow {}_{p^{\infty}} H^{i}(k , T(\bar{k})/T(p)).
$$
On the other hand, the quotient $T(\bar{k})/T(p)$ is a uniquely $p$-divisible group, implying that
$$
{}_{p^{\infty}} H^{i}(k , T(\bar{k})/T(p)) = 0 \ \ \text{for all} \ \  i \geqslant 1,
$$
and our claim for $i \geqslant 2$ follows. To consider the case $i = 1$, we need to show that the cokernel of the map
$$
H^0(k , T) = T(k) \stackrel{\alpha}{\longrightarrow} H^0(k , T(\bar{k})/T(p))
$$
is isomorphic to $T(k) \otimes_{\Z} \Q_p/\Z_p$. It is easy to see that $H^0(k , T(\bar{k})/T(p)) = T(p)_0/T(p)$, where
$$
T(p)_0 = \{t \in T(\bar{k}) \: \vert \: t^{p^m} \in T(k) \ \ \text{for some} \ \ m \geqslant 1 \}.
$$
So,  $\mathrm{Coker}\: \alpha \simeq T(p)_0/T(p)T(k)$.
But the map
$$
(t , \frac{n}{p^m}) \mapsto t^{n/p^m} \cdot T(p)T(k) \ \ \text{for} \ \ t \in T(k), \ n \in \Z
$$
extends to an isomorphism $T(k) \otimes_{\Z} \Z[1/p]/\Z \simeq T(p)_0/T(p)T(k)$, and assertion (2) follows.
\end{proof}

\vskip1mm

We will now describe the structure of $M(p)$. Let $\widetilde{C}$ be the smooth projective curve over $k$ that contains $C$ as an open subset. We then have the following localization sequence (cf. \cite{Jann}, p. 126):
\begin{equation}\label{E:XX011}
0 \longrightarrow H^1(\widetilde{C} \otimes_k \bar{k} , \Q_p/\Z_p(2)) \longrightarrow H^1(C \otimes_k \bar{k} , \Q_p/\Z_p(2)) \longrightarrow
\end{equation}
$$
\hfill \bigoplus_{x \in \widetilde{C}\setminus C} \mathrm{Ind}^{k}_{k(x)}(\Q_p/\Z_p(1)) \stackrel{tr}{\longrightarrow} \Q_p/\Z_p(1) \longrightarrow 0.
$$
Let us recall that given a finite collection $\ell_1, \ldots , \ell_d$ of finite separable extensions of an arbitrary field $k$, the $k$-torus
$$
T = \mathrm{Ker}\left( \prod_{i = 1}^d \mathrm{R}_{\ell_i/k}(\mathbb{G}_m) \stackrel{N}{\longrightarrow} \mathbb{G}_m \right),
$$
where $N$ is the product of the norm maps $N_{\ell_i/k}$ for the extensions $\ell_i/k$, $i = 1, \ldots , d$, is called the {\it multi-norm torus} associated with $\ell_1, \ldots , \ell_d$. Henceforth, we will assume that for each $x \in \widetilde{C} \setminus C$, the residue field $k(x)$ is a separable extension of $k$, which of course is automatically true if $k$ is perfect (recall that in Theorem \ref{T:char0}, $k$ is a number field).
%note that in the proof of Theorem \ref{T:H3-Finite}, $k$ is going to be a number field).
Then (\ref{E:XX011}) gives rise to the following exact sequence of Galois modules
\begin{equation}\label{E:XX012}
0 \to A(p) \longrightarrow M(p) \longrightarrow T(p) \to 0
\end{equation}
where $A(p) = H^1(\widetilde{C} \otimes_k \bar{k} , \Q_p/\Z_p(2))$ and $T(p) = {}_{p^{\infty}} T(\bar{k})$, with $T$ being the multi-norm torus associated with the field extensions $k(x)/k$ for $x \in \widetilde{C} \setminus C$.

\vskip2mm

Let $T$ be a torus over a field $\mathscr{K}$ which is complete with respect to a discrete valuation $v$. We call an element $x \in H^i(\mathscr{K} , T)$ \emph{unramified} if it lies in the image of the inflation map
$$
H^i(\mathscr{K}^{\mathrm{ur}}/\mathscr{K} , T(\mathscr{O}(\mathscr{K}^{\mathrm{ur}}))) \longrightarrow H^i(\mathscr{K} , T),
$$
where $\mathscr{K}^{\mathrm{ur}}$ is the maximal unramified extension of $\mathscr{K}$ with the valuation ring $\mathscr{O}(\mathscr{K}^{\mathrm{ur}})$, and
$T(\mathscr{O}(\mathscr{K}^{\mathrm{ur}}))$ is the (unique) maximal bounded subgroup of $T(\mathscr{K}^{\mathrm{ur}})$. It follows from \cite[Ch. II, \S7]{GMS} that for a finite unramified $\Ga(\overline{\mathscr{K}}/\mathscr{K})$-module $\mathscr{M}$ whose order is prime to the residue characteristic $\mathrm{char}\: \mathscr{K}^{(v)}$, an element $x \in H^i(\mathscr{K} , \mathscr{M})$ is unramified as defined in \S\ref{S:Pic} if and only if it lies in the image of the inflation map $H^i(\mathscr{K}^{\mathrm{ur}}/\mathscr{K} , \mathscr{M}) \longrightarrow H^i(\mathscr{K} , \mathscr{M})$. So, for a finite Galois submodule $\mathscr{M} \subset  T(\mathscr{K}^{\mathrm{ur}})$ of order prime to $\mathrm{char}\: \mathscr{K}^{(v)}$, the natural map $H^i(\mathscr{K} , \mathscr{M}) \to H^i(\mathscr{K} , T)$ takes unramified classes to unramified ones. In addition, if the splitting field of $T$ is unramified over $\mathscr{K}$, then any finite subgroup $\mathscr{M} \subset T(\overline{\mathscr{K}})$ of order prime to $\mathrm{char}\: \mathscr{K}^{(v)}$ is automatically contained in $T(\mathscr{O}(\mathscr{K}^{\mathrm{ur}}))$. More generally, for a torsion Galois submodule $\mathscr{M}$ of $T(\mathscr{O}(\mathscr{K}^{\mathrm{ur}}))$ that does not contain elements of order divisible by the residue characteristic, an element $x \in H^i(\mathscr{K} , \mathscr{M})$ is unramified if it comes from an unramified element in $H^i(\mathscr{K} , \mathscr{M}')$ for some finite Galois submodule $\mathscr{M}'$ of $\mathscr{M}$.

Now, if $T$ is a torus defined over an arbitrary field $k$ with a discrete valuation $v$, then $x \in H^i(k , T)$ is defined to be unramified at $v$ if its image in $H^i(k_v , T)$ is unramified in the sense specified above. Furthermore, for a set $V$ of discrete valuations of $k$, we let $H^i(k , T)_V$ denote the subgroup of $H^i(k , T)$ consisting of elements that are unramified at all $v \in V$. We will use these remarks to prove the following.
\begin{lemma}\label{L:XX015}
Let $k$ be a number field, let $\ell_1, \ldots , \ell_d$ be finite extensions of $k$, and let $T$ be the multi-norm torus associated with these extensions.
Suppose $p$ is a prime and let $V \subset V^k_f$ be a cofinite set of places such that $\mathrm{char}\: k^{(v)}$ is prime to $p$ for all $v \in V$. Then for any $m \geqslant 1$, the group $${}_{p^m} H^2(k , T(p))_V$$ is finite.
\end{lemma}
\begin{proof}
By shrinking $V$ if necessary, we may assume that the extensions $\ell_1, \ldots , \ell_d$ are unramified at all $v \in V$. Then it follows from the above discussion that the map $H^2(k , T(p)) \to H^2(k , T)$ takes unramified classes at $v \in V$ to unramified ones. Besides, according to Lemma \ref{L:XX007}, this map is injective. So, it is enough to show that for any $n \geqslant 1$, the group ${}_nH^2(k , T)_V$ is finite. We have an exact sequence of $k$-tori
\begin{equation}\label{E:XX016}
1 \to T \longrightarrow T_0 \stackrel{N}{\longrightarrow} \mathbb{G}_m \to 1 \ \ \text{where} \ \ T_0 = \prod_{i = 1}^d \mathrm{R}_{\ell_i/k}(\mathbb{G}_m).
\end{equation}
Since $H^1(k , \mathbb{G}_m) = 1$ by Hilbert's 90, the cohomological sequence associated with (\ref{E:XX016}) shows that the natural map $H^2(k , T) \to H^2(k , T_0)$ is injective. So, it is enough to establish the finiteness of ${}_nH^2(k , T_0)_V$, for which we only need to consider the case of $T_0 = \mathrm{R}_{\ell/k}(\mathbb{G}_m)$ for some finite (separable) extension $\ell/k$. As above, we may assume that every $v \in V$ is unramified in $\ell/k$. Let $V'$ be the set of all extensions to $\ell$ of places in $V$. Then the restriction of the isomorphism $H^2(k , T_0) \simeq H^2(\ell , \mathbb{G}_m)$ yields an isomorphism $H^2(k , T_0)_V \simeq H^2(\ell , \mathbb{G}_m)_{V'}$. So,
$$
{}_n H^2(k , T_0)_V \simeq {}_nH^2(\ell , \mathbb{G}_m)_{V'} = {}_n \Br(\ell)_{V'},
$$
the $V'$-unramified part of the $n$-torsion in the Brauer group of $\ell$, which is known to be finite (cf. \cite[3.5]{CRR3}).
\end{proof}

\vskip1mm

\noindent {\bf Remark 10.11.} Using the fact that the unramified Brauer group  ${}_n\Br(k)_V$ is finite for any finitely generated field $k$ of characteristic
prime to $n$ and a divisorial set of places $V$ (cf. \cite{CRR3}), one easily generalizes Lemma \ref{L:XX015} to any finitely generated field $k$ of characteristic $\neq p$.

\vskip2mm

\addtocounter{thm}{1}

\noindent {\it Proof of Proposition \ref{P:XX002}.} The exact sequence (\ref{E:XX012}) gives rise to the following exact sequence of cohomology groups
$$
H^1(k , T(p)) \stackrel{\alpha}{\longrightarrow} H^2(k , A(p)) \stackrel{\beta}{\longrightarrow} H^2(k , M(p)) \stackrel{\gamma}{\longrightarrow} H^2(k , T(p)).
$$
Clearly, $\gamma({}_{p^m} H^2(k , M(p))_V) \subset {}_{p^m} H^2(k , T(p))_V$, and the latter is finite by Lemma \ref{L:XX015}. So, it remains to show that the intersection  $$\mathrm{Im}\: \beta \cap {}_{p^m} H^2(k , M(p))$$ is finite. Pick $m_0 \geqslant 0$ so that $p^{m_0}$ annihilates ${}_{p^{\infty}} H^1(k , T)$ --- note that the latter has finite exponent dividing the degree $[\ell : k]$, where $\ell$ is the minimal splitting field of $T$. It follows from  the exact sequence in Lemma \ref{L:XX007}(2) that then the group $p^{m_0} \cdot H^1(k , T(p))$ is $p$-divisible.

Now, let $x \in H^2(k , A(p))$ be such that $p^m \cdot \beta(x) = 0$. Then $p^m x = \alpha(y)$ for some $y \in H^1(k , T(p))$, and consequently letting $d = m + m_0$ we have $p^d x = \alpha(p^{m_0} y)$. The divisibility of $p^{m_0} \cdot H^1(k , T(p))$ implies that one can find $z \in H^1(k , T(p))$ satisfying
$$
p^{m_0}y = p^d z.
$$
Then
$$
p^d(x - \alpha(z)) = 0 \ \ \text{and} \ \ \beta(x - \alpha(z)) = \beta(x).
$$
This proves the inclusion
$$
\beta({}_{p^d} H^2(k , A(p))) \supset \mathrm{Im}\: \beta \cap {}_{p^m} H^2(k , M(p)),
$$
and it remains to show that ${}_{p^d} H^2(k , A(p))$ is finite. But according to statement 1) on p. 127 of \cite{Jann}, the map
$$
H^2(k , A(p)) \longrightarrow \bigoplus_{v \in V^k} H^2(k_v , A(p))
$$
is an isomorphism. On the other hand, $H^2(k_v , A(p)) = 0$ for any $v \notin S' = S \cup V^k_{\infty}$, where $S$ is the set of points of bad reduction for $\widetilde{C}$ (\cite{Jann}, statement 5) on p. 131). Thus,
$$
{}_{p^d} H^2(k , A(p)) \simeq \bigoplus_{v \in S'} {}_{p^d} H^2(k_v , A(p)).
$$
But for any $v$, the group ${}_{p^d} H^2(k_v , A(p))$ is a quotient of $H^2(k_v , {}_{p^d} A(p))$, which is finite because ${}_{p^d} A(p)$ is finite and $k_v$ is a finite extension of $\Q_q$ for some $q$ (cf. \cite[Ch. II, \S 5, Proposition 14]{Serre-GC}),  completing the argument. \hfill $\Box$

\vskip1mm

(Note that while we actually prove the finiteness of $\mathrm{Im}\: \beta \cap {}_{p^m} H^2(k , M(p))$, for the proof of Proposition \ref{P:XX002}
it would be sufficient to prove the finiteness of $\mathrm{Im}\: \beta \cap {}_{p^m} H^2(k , M(p))_V$.)

\vskip2mm

\noindent {\bf 4. Proof of Theorem \ref{T:char0}.} First of all, it is enough to consider the case where $n = p^m$ with $p$ a prime and $m \geqslant 1$. Second, deleting a finite number of places from $V_1$ if necessary, we may assume that any $v \in V_1$ satisfies the assumptions made in the beginning of subsection 10.2 (recall that $\mathrm{char}\: k^{(v)}$ is prime to $p$ by the assumptions made in the statement of the theorem).
Consider the following commutative diagram
\begin{equation}\label{E:XX110}
\xymatrix{H^3(C, \mu_{p^m}^{\otimes 2})_V \ar[r]^{\alpha} \ar[d]_{\beta} & H^3(C , \Q_p/\Z_p(2)) \ar[d]^{\delta} \\ H^3(k(C) , \mu_{p^m}^{\otimes 2})_{V} \ar[r]^{\gamma} & H^3(k(C) , \Q_p/\Z_p(2))}
\end{equation}
Since $\gamma$ is injective (Lemma \ref{L:XX002}), it is enough to establish the finiteness of $\gamma(H^3(k(C) , \mu_{p^m}^{\otimes 2})_{V})$.
As we pointed out in subsection 4.2, the natural map $H^3(C , \mu_{p^m}^{\otimes 2}) \to H^3(k(C) , \mu_{p^m}^{\otimes 2})_{V_0}$ is surjective, implying that $\beta$ in (\ref{E:XX110}) is also surjective. Thus,
 $$
 \gamma(H^3(k(C) , \mu_{p^m}^{\otimes 2})_{{V}}) = \delta(\alpha(H^3(C , \mu_{p^m}^{\otimes 2})_V)),
 $$
 and it is enough to prove the following.
\begin{prop}\label{P:XX017}
The image of $\alpha \colon H^3(C , \mu_{p^m}^{\otimes 2})_V  \longrightarrow  H^3(C , \Q_p/\Z_p(2))$ is finite.
\end{prop}
\begin{proof}
We will use the following exact sequence
$$
H^3(k , \Q_p/\Z_p(2)) \stackrel{\iota^{2,3}_k(p)}{\longrightarrow} H^3(C , \Q_p/\Z_p(2)) \stackrel{\omega^{2,3}_k(p)}{\longrightarrow} H^2(k , M(p))
$$
obtained by taking the direct limit of the sequences (\ref{E:X0}) from \S\ref{S:ProofTF1}.2.
By Poitou-Tate (cf. \cite[Ch. II, 6.3]{Serre-GC}), for any $n \geqslant 1$, the natural map
$$
H^3(k , \mu_n^{\otimes 2}) \longrightarrow \prod_{v \in V^k_{\mathrm{real}}} H^3(k_v , \mu_n^{\otimes 2}),
$$
where $V^k_{\mathrm{real}}$ is the set of all real places of $k$, is an isomorphism. Complex conjugation acts on $\mu_n$ by inversion, hence acts trivially on $\mu_n^{\otimes 2}$. It follows that $H^3(k_v , \mu_n^{\otimes 2})$ is trivial if $n$ is odd, and
$$
H^3(k_v , \mu_2^{\otimes 2}) = H^1(k_v , \mu_2^{\otimes 2}) = \{ \pm 1 \}
$$
for any $v \in V^k_{\mathrm{real}}$. Thus, $H^3(k , \Q_p/\Z_p(2))$ is trivial if $p > 2$ and isomorphic to $(\Z/2\Z)^r$ for $p =2$, where $r = \vert V^k_{\mathrm{real}} \vert$. In any case, it is finite. So, in order to prove that $\alpha$ has finite image, it is enough to show that
$$
\Delta := \omega^{2 , 3}_k(p)(\alpha(H^3(C , \mu_{p^m}^{\otimes 2})_V))
$$
is finite. Clearly, $\Delta \subset {}_{p^m} H^2(k ,
M(p))$. Now, since the residue field $k^{(v)}$ is finite and the curve $C$  is affine, we have
$\mathrm{Pic}\: \uC^{(v)} \otimes_{\Z} \Q_p/\Z_p = 0$ (Corollary \ref{C:XX001}). So, applying Proposition \ref{P:unramX1}, we obtain that $\Delta$ is contained in  $H^2(k , M(p))_V$, and therefore in fact  $\Delta \subset {}_{p^m} H^2(k , M(p))_V$. Since the latter is finite by Proposition \ref{P:XX002}, the required fact follows.
\end{proof}

\vskip3mm

\noindent {\small {\bf Acknowledgements.} The first author was supported by the Canada Research Chairs Program and by an NSERC research grant. The second author was partially supported by the Simons Foundation. During the preparation of the final version of the paper, he visited Princeton University and the Institute for Advanced Study, and would like to thankfully acknowledge the hospitality of both institutions. The third author was partially supported by an AMS-Simons Travel Grant. We are grateful to P.~Gille and M.~Rapoport for useful conversations. Last but not least, we would like to express our gratitude to all anonymous referees  whose comments, insights and suggestions helped to improve the original paper significantly.}

\vskip8mm

\bibliographystyle{amsplain}

\end{document}